\documentclass{article}  
\pdfoutput=1
\usepackage{amsmath,amssymb,amsthm,amsfonts,graphicx,color,tikz,float,tikz-3dplot,subcaption}

\newtheorem{theorem}{Theorem}[section]
\newtheorem{definition}[theorem]{Definition}

\newtheorem{problem}[theorem]{Problem}
\newtheorem{proposition}[theorem]{Proposition}
\newtheorem{lemma}[theorem]{Lemma}
\newtheorem{corollary}[theorem]{Corollary}
\newtheorem{remark}[theorem]{Remark}
\newtheorem{example}[theorem]{Example}

\newtheorem*{theorem*}{Theorem}
\newtheorem*{corollary*}{Corollary}
\newtheorem*{proposition*}{Proposition}

\begin{document}
\author{Alexander Engstr\"om \and Florian Kohl}
\title{Transfer-Matrix Methods meet Ehrhart Theory}

\maketitle
{% additional braces for segregating \footnotesize
  \bigskip
  \footnotesize

  A.~Engstr\"om, \textsc{Aalto University
Department of Mathematics, P.O. Box 11100, FI-00076 Aalto, Finland}\par\nopagebreak
  \textit{E-mail address}, A.~Engstr\"om: \texttt{alexander.engstrom@aalto.fi}
  \medskip
  
  F.~Kohl, \textsc{Freie Universit\"at Berlin, FB Mathematik and Informatik,
Institut f\"ur Mathematik,Arnimallee 3, 14195 Berlin, Germany}\par\nopagebreak
  \textit{E-mail address}, F.~Kohl: \texttt{fkohl@math.fu-berlin.de}

}
\begin{abstract}
Transfer-matrix methods originated in physics where they were used to count the number of allowed particle states on a structure whose width $n$ is a parameter. Typically, the number of states is exponential in $n.$ One mathematical instance of this methodology is to enumerate the proper vertex colorings of a graph of growing size by a fixed number of colors.

In Ehrhart theory, lattice points in the dilation of a fixed polytope by a factor $k$ are enumerated. By inclusion-exclusion, relevant conditions on how the lattice points interact with hyperplanes are enforced. Typically, the number of points are (quasi-) polynomial in $k.$ The text-book example is that for a fixed graph, the number of proper vertex colorings with $k$ colors is polynomial in $k.$

This paper investigates the joint enumeration problem with both parameters $n$ and $k$ free. We start off with the classical graph colorings and then explore common situations in combinatorics related to Ehrhart theory. We show how symmetries can be explored to reduce calculations and explain the interactions with Discrete Geometry.
\end{abstract}

\section{Introduction}
\label{sec:Introduction}
Graph colorings have been intriguing mathematicians and computer scientists for decades. Historically, graph colorings first appeared in the context of the $4$-color conjecture. Birkhoff --- trying to prove said conjecture --- introduced what is now called the chromatic polynomial. Whitney later generalized this notion from planar graphs to arbitrary graphs, see \cite{Whitneychromatic}. Chromatic polynomials are one of the fundamental objects in algebraic graph theory with many questions about them still unanswered. For instance, in 1968 Read asked which polynomials arise as chromatic polynomials of some graph. This question remains wide open to this day. However, some progress has been made. In 2012, June Huh showed that the absolute values of the coefficients form a log-concave sequence, see \cite[Thm. 3]{Huh}, thus proving a conjecture by Rota, Heron, and Welsh. Not only classifying chromatic polynomials is extremely challenging, but also explicitly computing the coefficients turns out be $\#P$-hard, see \cite{JaegerVertigan}.

In the first part of this article, we examine proper $k$-colorings of Cartesian graph products of the form $G \times P_n$ and $G \times C_n$, where $G$ is an arbitrary graph and $P_n$ ($C_n$) is the path (cycle) graph on $n$ nodes, respectively. The motivation for this problem is twofold. 

First, it lies at the intersection of transfer-matrix methods and Ehrhart theory, both areas being interesting in their own right. Classically, transfer-matrix methods have been used to count the number of (possibly closed) walks on weighted graphs. However, transfer-matrix methods also made an appearance in seemingly unrelated areas such as calculating DNA-protein-drug binding in gene regulation \cite{Teif}, the $3$-dimensional dimer problem \cite{Ciucu}, counting graph homomorphisms \cite{LundowMarkstrom}, computing the partition function for some statistical physical models \cite{Freedman}, and determining the entropy in physical systems \cite{Friedland}. One of the big problems in these applications is that the size of the transfer matrices increases extremely fast as the size of the system increases. Therefore, one needs to either limit the size of the system or find a way of ``compactifying" these transfer matrices. In \cite{Ciucu}, Ciucu uses symmetry to reduce the size of the matrix. We follow and expand these ideas. Similar techniques have also been used by \cite{LundowMarkstrom}. Ehrhart theory is the study of integer points in polytopes and as such it appears in various disguises anytime someone tries to examine/count/find Diophantine solutions to a system of linear inequalities with bounded solution set. Moreover, it is related to algebraic geometry and commutative algebra \cite{CLS, StanleyGreenBook,BrunsHerzog}, optimization  \cite{BeckPixton,sturmfels1996}, number theory \cite{BeckEtAl-GorensteinLHC, BeckKohl, Pommersheim93}, combinatorics \cite{StanleyGreenBook, CCD}, and --- for this article most importantly --- 	to proper graph colorings \cite{BeckZaslavsky06}.

Second, this problem also has direct applications to physics: If $G$ represents a molecular structure, then $G \times P_n$ corresponds to several connected layers of that molecular structure. The $k$ colors correspond to $k$ different states of the atoms. Counting the number of possible combinations is the same as counting the number of colorings. If we furthermore assume that two adjacent atoms are not allowed to be in the same state, we arrive at a classical proper coloring problem. Since $n$ is very large in physical systems and also $k$ may vary, the (doubly or single) asymptotic behavior is of interest.

Our work combines transfer-matrix methods with Ehrhart theory. As an intermediate step, we examine proper colorings of a graph, where some nodes are already colored. We call these colorings \emph{restricted colorings}. The associated counting function is a polynomial and it satisfies a reciprocity statement.
\begin{theorem*}
[Theorem~\ref{thm:restrictedReciprocity}]
Let $\Gamma = (\{1,2,\dots,n \},E)$ be a graph and fix a proper $k'$-coloring $c' \colon V' \rightarrow \{1,2,\dots,k' \}$ on the induced subgraph $\left. G\right|_{V'}$ for a subset $V' \subset V(\Gamma)$. Then, for $k \geq k'$, the restricted chromatic polynomial 
\begin{equation}
\label{eq:RestrictedChromaticPolynomial}
\chi_{c', \Gamma}(k)= \# \text{proper $k$-colorings c of }\Gamma \text{ such that }\left.c \right|_{V'} = c'.
\end{equation} is a polynomial of degree $\#V - \#V'=:s$ with leading coefficient $1$, whose coefficients $a_i$ alternate in sign, and whose absolute values of the coefficient form a \emph{log-concave sequence}, i.e., $a_i^2 \geq a_{i-1}a_{i+1}$ holds for $0 < i<s$. The second highest coefficient $a_{n-1}$ is given by
\[
-a_{n-1} =\# \text{edges }\{ v_i, v_j\} \text{ such that } \{v_i, v_j \}\nsubseteq V'.
\]
Moreover, we have the reciprocity statement
\begin{align*}
\label{eq:RestrictedReciprocity}
\chi_{c', \Gamma}(-k) &= (-1)^{s} \# (\alpha, c)\text{ of }\Gamma \text{ with } \left.c \right|_{V'}=c'\\
					 &= (-1)^s \chi_{c',\Gamma}(k),\\
\end{align*}
where $(\alpha,c)$ is a pair of an acyclic orientation $\alpha$ and a compatible (\emph{not necessarily proper}) $k$-coloring, $c$, of $\Gamma$. 
\end{theorem*}

In Definition~\ref{def:DefinitionL}, using a group action and restricted $k$-colorings, we define a compactified transfer matrix $L$. The rows and columns of $L$ are labeled by orbits $o_1$, $o_2$, $\dots$, $o_p$ of this group action, see Definition~\ref{def:orbits}. As it turns out, all entries of $L$ are polynomials in $k$:
\begin{theorem*}[Theorem~\ref{thm:BasicGeometricFacts}]
With the notation from above and with $k\geq N$, we have:
\begin{enumerate}
\item Every entry $L_{o_i, o_j}(k)$ equals the sum of Ehrhart polynomials of lattice inside-out polytopes of dimension $\# \text{colors of }o_j$ and hence is a polynomial of degree $\# \text{colors of }o_j$,
\item $L_{o_i, o_j}(k)$ is independent of the choice of the representative, i.e., it is well-defined,
\item $\# o_i \cdot L_{o_i, o_j}(k) = \chi_{G_{o_i,o_j}}(k)$, where
\begin{align*}
\chi_{G_{o_i, o_j}}(k) = &\# \text{proper } k\text{-colorings}\colon \text{coloring of }G \times \{1\} \text{ is in }o_i \text{ and } \\ &\text{ coloring of }G \times \{2\} \text{ is in }o_j.\\
\end{align*}
\end{enumerate}
\end{theorem*}

The matrix $L$ in fact behaves like a transfer matrix:
\begin{theorem*}[Theorem~\ref{thm:L^n_entry}]
Let $V(P_{n+1}) = \{1,2,\dots, n+1\}$ and let $G$ be any graph. Let $o_1, o_2,\dots, o_p$ be the orbits as defined in Definition \ref{def:orbits}. Then, for $k \geq \# \text{colors used in }o_i$, the $(o_i,o_j)$-entry of $L^n$  counts the number of proper $k$-colorings of $G \times P_{n+1}$, where $G \times \{1\}$ is fixed by a coloring in orbit $o_i$, and where the coloring of $G \times \{n+1\}$ lies $o_j$. 
\end{theorem*}

Moreover, $L$ can be used to explicitly compute the chromatic polynomial of $G \times P_n$.
\begin{corollary*}[Corollary~\ref{cor:ChromPolyProduct}]
Let $G \times P_{n+1}$ and $L$ be as above. Then
\begin{equation}
\label{eq:chromatic_poly_P_n}
\chi_{G \times P_{n+1}}(k) = (w_1(k),\dots, w_p(k)) L^n \mathbf{1},
\end{equation}
where $w_i (k)$ is the size of $o_i$ and $\mathbf{1}: = (1,\dots, 1)^t$.
\end{corollary*}

The row sums of this matrix also satisfy a reciprocity theorem:
\begin{theorem*}[Theorem~\ref{thm:Reciprocity_L^n}]
Let $L\in \mathbb{Z}^{p \times p}$ be as above, let $L^n_{i} : = \sum_{k=1}^p \left(L^n\right)_{i,k}$ be the $i^{\text{th}}$ row sum of $L^n$, and let $V(P_{n+1}) = \{1,2,\dots, n+1 \}$. Then, for $k \geq N=\# V(G)$, we have
\begin{equation}
\label{eq:reciprocity_L^n}
L^n_{i} (-k) =(-1)^{Nn} \# (\alpha, c)\text{ of } G \times P_{n+1} \text{ where } G \times \{ 1\} \text{ is colored by repr. of } o_i,
\end{equation}
where $(\alpha,c)$ is a pair of an acyclic orientation $\alpha$ and a compatible $o_i$-restricted $k$-coloring $c$.
\end{theorem*}

Asymptotically, the power of the biggest eigenvalue $\lambda_{\max}^{n-1}$ of $L$ determines the number of proper colorings of $G\times C_n$. We give explicit bounds of this eigenvalue in terms of the row sums. Let $\delta (L )$  and $\Delta (L)$ be the smallest and biggest row sums of $L$, respectively.

\begin{proposition*}[Proposition~\ref{prop:AsymptoticsCn}]
Let $G$ be a graph and $N = \# V(G)$ and let $\delta (L)$ and $\Delta (L)$ be as above. Then the doubly asymptotic behavior of the number of proper $k$-colorings of $G \times C_n$ is dominated by $\lambda_{\max}$ and
\[
\delta(L) \leq \lambda_{\max} \leq \Delta(L),
\]
where $\delta(L)  = \sum_{i=0}^N a_i k^i $, $\Delta(L)=\sum_{i=0}^N b_i k^i $, $a_N = b_N$, and $a_{N-1} = b_{N-1}$.
\end{proposition*}

In the second part of this note, we show that a similar philosophy can be used to count the number discrete Markov chains. In (\ref{eq:MarkovMatrixDef}), we define a matrix $M$ that also acts as a transfer matrix. 

\begin{theorem*}[Theorem~\ref{thm:MarkovMatrix}]
With the previous notation, we have:

\begin{itemize}
\item the number of chains of length $n+1$ is given by
\begin{equation}
\label{eq:MarkovChain}
\# \text{number of chains of length n+1} = \mathbf{1}^t M^n \left|_{x = (1,1,\dots,1)} \right. \mathbf{1}\text{,}
\end{equation}
\item the number $I_k$ of chains so that no element is increased by more than $k$ is given by
\begin{equation}
\label{eq:MarkovLocalChange}
I_k = \left(\deg_{x_1, x_2, \dots, x_r\leq k}(x^{b^1},\dots,x^{b^n}) M^n \right) \left|_{x = (1,1,\dots,1)} \right. \mathbf{1}\text{,}
\end{equation}
where no indeterminate $x_i$ can have a degree bigger than $k$.
\end{itemize} 
\end{theorem*}

We show that this method can be used to count the number of order-preserving maps from a certain class of posets --- which we call stacked posets --- into $[k]: = \{1,2\dots,k \}$.

\begin{proposition*}[Proposition~\ref{SurjOrderPreserving}]
Let $\mathcal{P}_n$ be a stacked poset. The number of surjective, order-preserving maps from $\mathcal{P}_n$ into $[k]$ is given by
\begin{equation}
\label{eq:SurjOrderPreserving}
\# \left\{\pi\colon \mathcal{P}_n \twoheadrightarrow [k]  \colon \text{order-pres.} \right\}= \left(\operatorname{tdeg}_{= k} (x^{b^1},\dots, x^{b^r})M^{n}\right)\left|_{x = (1,1,\dots,1)}, \right. \mathbf{1}
\end{equation}
where $x = (x_1, \dots, x_m)$ and where $\operatorname{tdeg}_{= k}$ denotes the terms whose total degree equals $k$.
\end{proposition*}

This article is structured as follows. In Section~\ref{sec:Background}, we introduce some basic notions about graphs. In particular, we introduce proper colorings, acyclic orientations, and the Cartesian graph product. In Section~\ref{sec:Special Cases}, we introduce a transfer-matrix method and show how one can use this to count the number of proper colorings when the number of colors is fixed. We then --- following \cite{BeckZaslavsky06} --- introduce inside-out polytopes and show how to count proper colorings of $G \times P_n$ when $n$ is fixed. In Section~\ref{sec:TMMEhrhart}, we illustrate how one can use symmetry to define a compactified transfer matrix $L$, whose size does not depend on $n$. We state and prove our main results about this matrix $L$. We end this section, with a brief interlude on counting the number orbits under a group action, where Bell numbers make a surprising appearance. We then switch gears for Section~\ref{sec:DMP}, where we show that a similar approach can be used to count discrete Markov chains. 

\section{Acknowledgement}
We would like to thank the anonymous referee for helpful suggestions and constructive criticism. The first author wants thank G\"unter M. Ziegler and FU Berlin for their hospitality during his sabbatical, which is where this work began. The second author was supported by a scholarship of the Berlin Mathematical School. He would like to thank the first author and Aalto University for their hospitality and their support during a month-long research stay. He would also like to thank Christian Haase for fruitful discussions and his support.

\section{Background and Notation}
\label{sec:Background}
This section reviews basic concepts about graphs such as proper $k$-colorings, the chromatic polynomial, and the Cartesian product of graphs. We end this section by stating two explicit counting problems that we address in Section \ref{sec:TMMEhrhart}, see Problem \ref{prob:ColoringsPnandCn}.

For $m\in \mathbb{Z}_{>0}$, we set $[m]:= \{1,2,\dots,m \}$. A \emph{graph $G$} is a pair $(V,E)=(V(G),E(G))$, where $V$ is the set of \emph{nodes} or \emph{vertices} and $E$ is the set of \emph{edges}. We assume all our graphs to be finite and we globally set $N: = \# V(G)$.  Moreover, we define the \emph{path graph $P_n$} to be the graph on the vertex set $[n]$ with edges $\{i,i+1\}$ for $i \in [n]$. We also define the \emph{cycle graph $C_n$} to be the graph with vertex set $[n]$ and with edge $\{i,i+1\}$ and $\{1,n \}$. A \emph{graph automorphism of a graph $G = (V,E)$} is a permutation $\sigma$ of the vertex set such that $\{i,j \}$ is an edge if and only if $(\sigma(i),\sigma(j))$ is an edge. The set of automorphisms of a graph $G$ together with the composition operation forms a group, which is called the \emph{automorphism group of $G$}.

This article mainly focuses on a special family of graphs, namely the Cartesian product of an arbitrary graph $G$ with either the path graph $P_n$ or the cycle graph $C_n$. 

\begin{definition}
\label{def:CartesianProduct}
Let $G_1 = (V(G_1),E(G_1))$ and $G_2 = (V(G_2),E(G_2))$ be two graphs. The \emph{Cartesian product} $G_1 \times G_2$ (sometimes in the literature also denoted $G_1 \square G_2$) is the graph with vertex set $V(G_1) \times V(G_2)$, and vertices $(u_1,v_1)$ and $(u_2,v_2)$ are connected by an edge if  
\begin{itemize}
\item either $u_1 = u_2$ and $\{v_1,v_2 \} \in E(G_2)$,
\item or if $v_1 = v_2$ and $\{u_1, u_2 \} \in E(G_1)$.
\end{itemize}
\end{definition}
Figure \ref{fig:GraphProduct} illustrates Definition \ref{def:CartesianProduct}. 
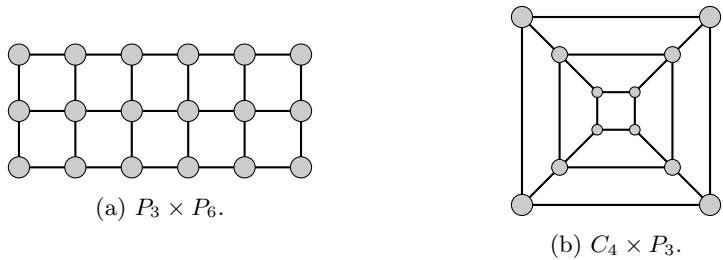
\begin{figure}[t]
%\label{fig:GraphProduct}
\center
\begin{subfigure}{0.5\textwidth}
\centering
\begin{tikzpicture}[darkstyle/.style={circle,draw,fill=gray!40,minimum size=20}, scale = .5]
\draw[step=1.5,thick] (0,-3) grid (7.5,0);
\foreach \x in {0,...,5}
{   \foreach \y in {0,...,2}
    {   \pgfmathtruncatemacro{\nodelabel}{\x+\y*5+1}
        \node[circle,draw=black,fill=white!80!black,minimum size=20, scale=.4] (\nodelabel) at (1.5*\x,-1.5*\y) { };
    }
}
\end{tikzpicture}
\caption{$P_3 \times P_6$.}
\end{subfigure}%
\begin{subfigure}{0.5\textwidth}
\centering
\begin{tikzpicture}[scale = 0.5]
\node[circle,draw=black,fill=white!80!black,minimum size=20, scale=.4] (1) at (1.25,2) { };
\node[circle,draw=black,fill=white!80!black,minimum size=20, scale=.4] (2) at (6.25,2) { };
\node[circle,draw=black,fill=white!80!black,minimum size=20, scale=.4] (3) at (6.25,7) { };
\node[circle,draw=black,fill=white!80!black,minimum size=20, scale=.4] (4) at (1.25,7) { };
\draw[thick] (1)--(2)--(3) -- (4)--(1);

\node[circle,draw=black,fill=white!80!black,minimum size=20, scale=.3] (5) at (2.25,3) { };
\node[circle,draw=black,fill=white!80!black,minimum size=20, scale=.3] (6) at (5.25,3) { };
\node[circle,draw=black,fill=white!80!black,minimum size=20, scale=.3] (7) at (5.25,6) { };
\node[circle,draw=black,fill=white!80!black,minimum size=20, scale=.3] (8) at (2.25,6) { };
\draw[thick] (5) -- (6) -- (7) -- (8) -- (5);

\node[circle,draw=black,fill=white!80!black,minimum size=20, scale=.2] (9) at (3.25,4) { };
\node[circle,draw=black,fill=white!80!black,minimum size=20, scale=.2] (10) at (4.25,4) { };
\node[circle,draw=black,fill=white!80!black,minimum size=20, scale=.2] (11) at (4.25,5) { };
\node[circle,draw=black,fill=white!80!black,minimum size=20, scale=.2] (12) at (3.25,5) { };
\draw[thick] (9) -- (10) -- (11) -- (12) -- (9);

\draw[thick] (1) -- (5) -- (9);
\draw[thick] (2) -- (6) -- (10);
\draw[thick] (3) -- (7) -- (11);
\draw[thick] (4) -- (8) -- (12);
\end{tikzpicture}
\caption{$C_4 \times P_3$.}
\end{subfigure}
\caption{The Cartesian graph product illustrated.}
\label{fig:GraphProduct}
\end{figure}

A \emph{proper $k$-coloring of a graph $G = (V,E)$} is a map $c \colon V \longrightarrow [k]$ such that $c(v) \neq c(u)$ for all $u$, $v$ with $\{ u,v \} \in E$. We refer to the elements of $[k]$ as \emph{colors}. 
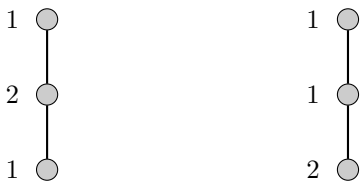
\begin{figure}[H]
\center
\begin{tikzpicture}[darkstyle/.style={circle,draw,fill=gray!40,minimum size=20}, scale = .5]
\node[circle,draw=black,fill=white!80!black,minimum size=20, scale=.4] (1) at (0,4) { };
\node[circle,draw=black,fill=white!80!black,minimum size=20, scale=.4] (2) at (0,2) { };
\node[circle,draw=black,fill=white!80!black,minimum size=20, scale=.4] (3) at (0,0) { };
\draw[thick] (1)--(2)--(3);
\node[left] at (-.5,4) {$1$};
\node[left] at (-.5,2) {$2$};
\node[left] at (-.5,0) {$1$};
%%%%%%%%%
\node[circle,draw=black,fill=white!80!black,minimum size=20, scale=.4] (5) at (8,4) { };
\node[circle,draw=black,fill=white!80!black,minimum size=20, scale=.4] (6) at (8,2) { };
\node[circle,draw=black,fill=white!80!black,minimum size=20, scale=.4] (7) at (8,0) { };

\draw[thick] (5)--(6)--(7);

\node[left] at (7.5,4) {$1$};
\node[left] at (7.5,2) {$1$};
\node[left] at (7.5,0) {$2$};
\end{tikzpicture}
\caption{A proper $2$-coloring and a non-proper $2$-coloring of $P_3$.}
\end{figure}

We would like to count proper $k$-colorings. For a graph $G$, we define the \emph{chromatic polynomial of $G$} as the counting function
\[
\chi_G(k) = \# \text{ proper } k\text{-colorings of }G.
\]
As the name suggests, $\chi_G$ agrees with a polynomial. 
\begin{theorem}[{\cite{Whitneychromatic}}]
Let $G$ be a simple graph on $n$ vertices. Then $\chi_G$ is a monic polynomial in $k$ of degree $n$.
\end{theorem} 
This polynomial extends the domain of $\chi$ from $\mathbb{Z}_{>0}$ to $\mathbb{R}$. Therefore, it is natural to ask whether other evaluations also have a combinatorial interpretation. Stanley \cite{Stanley12} famously gave an interpretation for evaluations at negative integers by relating it to acyclic orientations. An \emph{acyclic orientation} of a graph $G$ is an orientation of the edges such that the directed graph does not contain any cycles. 

\begin{figure}[H]
\center
\begin{tikzpicture}[darkstyle/.style={circle,draw,fill=gray!40,minimum size=20}, scale = .5]
\node[circle,draw=black,fill=white!80!black,minimum size=20, scale=.4] (1) at (4,0) { };
\node[circle,draw=black,fill=white!80!black,minimum size=20, scale=.4] (2) at (2,2) { };
\node[circle,draw=black,fill=white!80!black,minimum size=20, scale=.4] (3) at (0,0) { };
\draw[thick] (1)--(2)--(3) -- (1);

%%%%%%%%%
\node[circle,draw=black,fill=white!80!black,minimum size=20, scale=.4] (5) at (6,0) { };
\node[circle,draw=black,fill=white!80!black,minimum size=20, scale=.4] (6) at (10,0) { };
\node[circle,draw=black,fill=white!80!black,minimum size=20, scale=.4] (7) at (8,2) { };

\draw[->, thick] (5)--(6);
\draw[->, thick] (7)--(6);
\draw[->, thick] (5)--(7);

\node[circle,draw=black,fill=white!80!black,minimum size=20, scale=.4] (8) at (12,0) { };
\node[circle,draw=black,fill=white!80!black,minimum size=20, scale=.4] (9) at (16,0) { };
\node[circle,draw=black,fill=white!80!black,minimum size=20, scale=.4] (10) at (14,2) { };

\draw[->, thick] (8)--(9);
\draw[->, thick] (9)--(10);
\draw[->, thick] (10)--(8);

\end{tikzpicture}
\caption{$G$, an acyclic orientation of $G$, and a cyclic orientation of $G$.}
\end{figure}
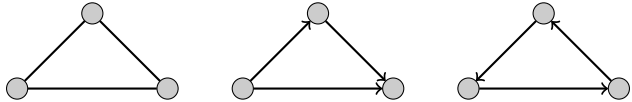

Given an acyclic orientation $\alpha$ of a graph $G$ and given a (not necessarily proper $k$-coloring) $c$, we say that $(\alpha, c)$ are compatible if $c_j \geq c_i$, where $c_l := c(l)$, whenever the orientation of the edge $\{i,j \}$ is $(i,j)$.  We give Beck's and Zaslavsky's reformulation of \cite[Theorem 1.2]{Stanley18}:

\begin{proposition}\cite[Cor. 5.5]{BeckZaslavsky06}
\label{prop:ClassicChromaticReciprocity}
The number of pairs $(\alpha,c)$ consisting of an acyclic orientation of an ordinary graph $G$ and a compatible $k$-coloring equals $(-1)^n \chi_G(-k)$, where $n$ is the number of vertices of $G$. In particular, $(-1)^n \chi_G(-1)$ equals the number of acyclic orientations of $G$. 
\end{proposition}

We will later come across what we call a \emph{$c'$-restricted}, proper $k$-coloring. Let $\Gamma= (V,E)$ be a graph and fix a proper $k'$-coloring $c' \colon V' \rightarrow \{1,2,\dots,k' \}$ on the \emph{induced subgraph $\left.\Gamma \right|_{V'}$} of a subset $V' \subset V(\Gamma)$, i. e., a proper $k'$-coloring on the graph $\left.\Gamma \right|_{V'}=(V',E')$, where $E' = \{\{i,j\}\colon i ,j \in V', \{i,j \} \in E \}$. Then, for $k \geq k'$, we define the \emph{$c'$-restricted chromatic polynomial }
\begin{equation*}
\chi_{c', \Gamma}(k)= \# \text{proper $k$-colorings c of }\Gamma \text{ such that }\left.c \right|_{V'} = c'.
\end{equation*}
This name will be justified in Theorem \ref{thm:restrictedReciprocity}, where we will also state a restricted analogue of Proposition \ref{prop:ClassicChromaticReciprocity}. A not necessarily proper $c'$-restricted $k$-coloring is \emph{compatible} with an acyclic orientation $\alpha$ if 
\begin{enumerate}
\item $(\alpha, c)$ are compatible in the usual sense,
\item and if $v\in V \setminus V'$ is adjacent to $u_1$, $u_2$, $\dots$, $u_s \in V'$ with $c'(u_1)=\dots = c'(u_s)$, then the orientations of the edges $\{v,u_i \}$ have to be the same for all $i$. 
\end{enumerate}

We are now ready to state the main problems of the first part of this article:
\begin{problem}
\label{prob:ColoringsPnandCn}
\begin{enumerate}
\item How many proper $k$-colorings does the Cartesian graph product $G \times P_n$ have?
\item Asymptotically, how many proper $k$-colorings does the Cartesian graph product $G \times C_n$ have?
\end{enumerate}
\end{problem}
We would like to remark that both $k$ and $n$ are variables. Hence, there are two prominent special cases which we will address in Section \ref{sec:Special Cases}:
\begin{enumerate}
\item $k$ is fixed and $n$ varies (see Section \ref{sec:TMM}),
\item $n$ is fixed and $k$ varies (see Section \ref{sec:InsideOut}).
\end{enumerate}
\section{Prominent Special Cases}
\label{sec:Special Cases}
\subsection{Transfer-Matrix Methods}
\label{sec:TMM}
In this section, we want to introduce a transfer-matrix method and apply it to Problem~\ref{prob:ColoringsPnandCn} in the case where $k$ is fixed. We follow the notation and the results of \cite[Section 4.7]{Stanley12}. 

One of the most basic applications of transfer-matrix methods is counting the number of walks in a given graph. A \emph{walk $\Gamma$ of length $n$}  from node $u$ to node $v$ is a sequence $u=v_0 e_1v_1e_2v_2 \dots e_n v_n= v$ of edges and vertices such that the endpoints of $e_i$ are $v_{i-1}$ and $v_i$. By abuse of notation, we will abbreviate this notation to $e_1e_2\dots e_n$. A walk where $u=v$ is called a \emph{loop}. Let $w\colon E \rightarrow \mathbb{C}$ be a weight function. If $\Gamma=e_1 e_2 \dots e_n$ is a walk, then the \emph{weight of $\Gamma$} is defined as $w(\Gamma) = \prod_{i \in [n]} w(e_i)$.
We will now define a matrix $A(n)$ by
\[
A_{i,j}(n) := \sum_{\Gamma} w(\Gamma),
\]
where the sum ranges over all walks of length $n$ from vertex $v_i$ to vertex $v_j$. Moreover, we define another matrix $A$ by
\[
A_{i,j} := \sum_{e \in E} w(e),
\]
where the sum ranges over all edges going from vertex $v_i$ to vertex $v_j$.
$(A)_{i,j}$ is called the \emph{adjacency matrix of $G$}, but we will sometimes call it the \emph{transfer matrix of $G$}. The following theorem forms one of the building blocks of transfer-matrix methods:
%%%%%%%%%%%%%%%%%%%%%
\begin{theorem}\cite[Thm. 4.7.1]{Stanley12}\\
Let $n\in \mathbb{Z}_{>0}$. Then $A^n_{i,j} = A_{i,j}(n)$, where we define $A^0 := I$ and where $A^n_{i,j}$ is the $(i,j)$-entry of the matrix $A^n$.
\end{theorem}
%%%%%%%%%%%%%%%%%%%%%%
\begin{remark}
In particular, if $w \equiv 1$, then $A^n_{i,j}$ counts the number of walks from $v_i$ to $v_j$ of length $n$. Moreover, the number of \emph{closed} walks is counted by trace of $A^n$. Since trace is the sum of the eigenvalues, and since the biggest eigenvalue $\lambda_{\max}$ is positive by the Perron--Frobenius theorem, the trace is asymptotically dominated by $\lambda_{\max}^n$.
\end{remark}
Now let us fix $k$ and apply transfer-matrix methods to Problem \ref{prob:ColoringsPnandCn}. Let $G$ be a graph with vertex set $V(G) = \left\{v_1, \dots, v_p \right\}$. To count the number of proper $k$-colorings of $G \times P_n$, we associate a new graph $M_G$ to $G$. We define $M_G$ to be the graph with
\begin{itemize}
\item vertex set $\mathcal{C}$, where $\mathcal{C}$ is the set of proper $k$-colorings of $G$,
\item and where two vertices $(c_1(v_1),,\dots,c_1(v_p))$ and $(c_2(v_1), \dots,c_2(v_p))$ are connected if $c_1(v_i)\neq c_2(v_i)$ for all $i \in [p]$.
\end{itemize}
Figure \ref{fig:TMM Graph} shows $M_{P_3}$ for $3$ colors. 

\begin{figure}[H]
\center
\begin{tikzpicture}[scale=.5]
\node[circle,draw=black,fill=white!80!black,minimum size=20, scale=.75]
        (1) at (3,9) {(1,2,1)};
\node[circle,draw=black,fill=white!80!black,minimum size=20, scale=.75]
        (2) at (7,7) {(1,3,1)};
\node[circle,draw=black,fill=white!80!black,minimum size=20, scale=.75]
        (3) at (9,3) {(2,1,2)};

\node[circle,draw=black,fill=white!80!black,minimum size=20, scale=.75]
        (4) at (-3,9) {(2,3,2)};
\node[circle,draw=black,fill=white!80!black,minimum size=20, scale=.75]
        (5) at (-7,7) {(3,1,3)};
\node[circle,draw=black,fill=white!80!black,minimum size=20, scale=.75]
        (6) at (-9,3) {(3,2,3)};

\node[circle,draw=black,fill=white!80!black,minimum size=20, scale=.75]
        (7) at (-3,-9) {(1,2,3)};
\node[circle,draw=black,fill=white!80!black,minimum size=20, scale=.75]
        (8) at (-7,-7) {(1,3,2)};
\node[circle,draw=black,fill=white!80!black,minimum size=20][scale=.75]
        (9) at (-9,-3) {(2,1,3)};

\node[circle,draw=black,fill=white!80!black,minimum size=20, scale=.75]
        (10) at (3,-9) {(2,3,1)};
\node[circle,draw=black,fill=white!80!black,minimum size=20, scale=.75]
        (11) at (7,-7) {(3,1,2)};
\node[circle,draw=black,fill=white!80!black,minimum size=20, scale=.75]
        (12) at (9,-3) {(3,2,1)};

\draw (1) -- (4);
\draw (1) -- (3);
\draw (1) -- (9);
\draw (1) -- (11);
\draw (1) -- (5);

\draw (2) -- (6);
\draw (2) -- (9);
\draw (2) -- (11);
\draw (2) -- (3);
\draw (2) -- (5);

\draw (3) -- (7);
\draw (3) -- (6);
\draw (3) -- (12);

\draw (4) -- (5);
\draw (4) -- (6);
\draw (4) -- (7);
\draw (4) -- (12);

\draw (5) -- (10);
\draw (5) -- (8);

\draw(6) -- (10);
\draw (6) -- (8);

\draw (7) -- (10);
\draw (7) -- (11);

\draw (8) -- (9);
\draw (8) -- (12);

\draw (9) -- (12);

\draw (10) -- (11);
\end{tikzpicture}
\caption{$M_{P_3}$.}
\label{fig:TMM Graph}
\end{figure}

This construction establishes a connection between $k$-colorings of $G \times P_n$ ($G \times C_n$) and walks (closed walks) of length $n$ in $M_G$, respectively. By abuse of notation, let $u^1u^2\dots u^n$ be a walk on $M_G$. By construction, this walk corresponds to the proper $k$-coloring of $G\times P_n$, where $G \times \{ i\}$ is colored by $u^i$ for all $i$. Similarly, if the walk is closed and thus $u^1 = u^n$, we get a corresponding proper $k$-coloring of $G \times C_{n-1}$. Therefore, we can count the number of $k$-colorings of $G \times P_n$ by computing powers of the adjacency matrix $A_{M_G}$ of $M_G$. Moreover, we can asymptotically count the number of colorings of $G\times C_n$ by analyzing the biggest eigenvalue of $A_{M_G}$.

The size of the transfer matrix $A_{M_G}$ is $\chi_G(k) \times \chi_G(k)$, since the number of vertices of $M_G$ is the number of proper $k$-colorings of $G$. Figure \ref{fig:bigMatrix} shows $A_{M_{P_3}}$ for $4$ colors.
\begin{figure}[H]
\includegraphics[width=\textwidth]{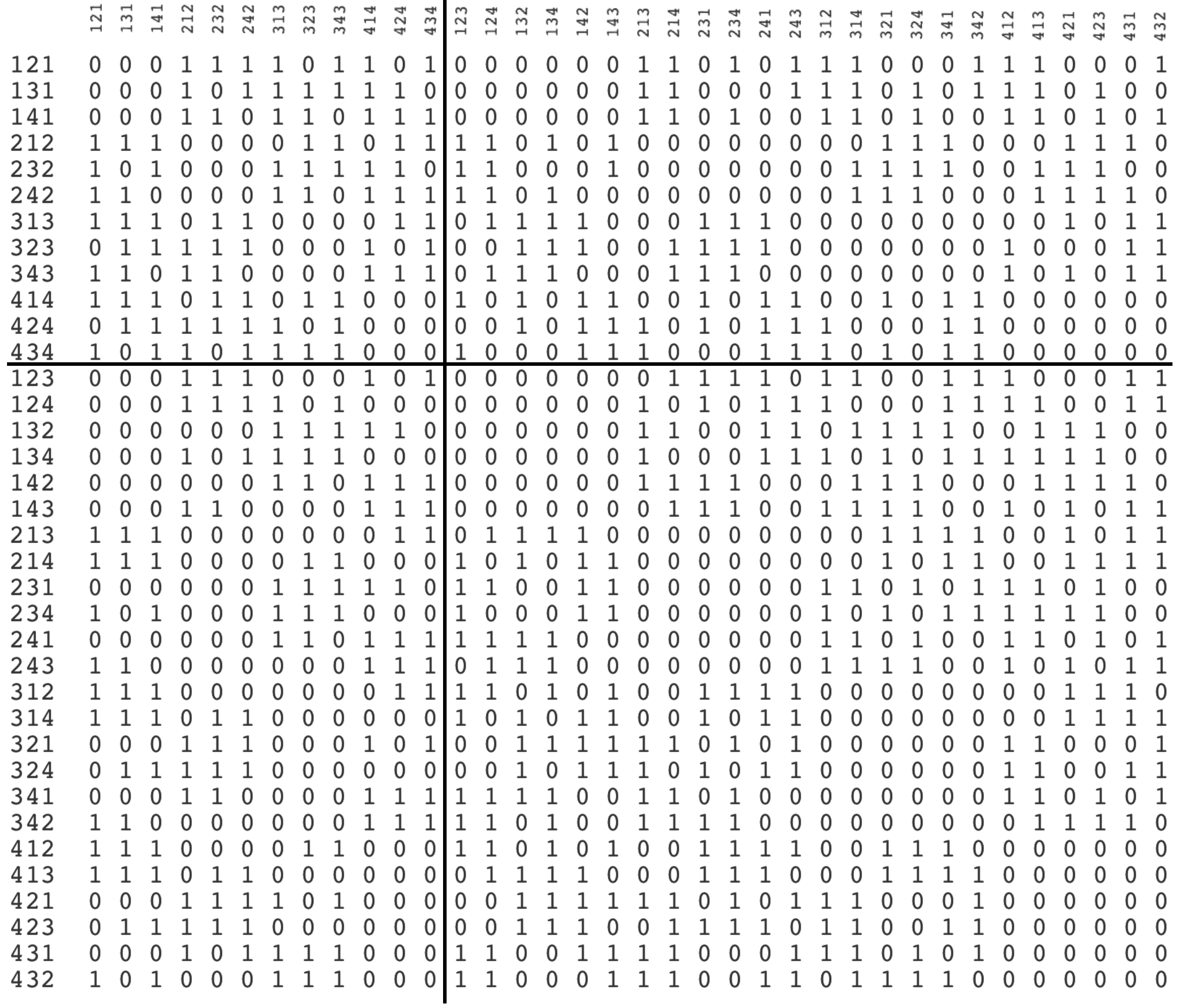}
\caption{Adjacency matrix of $M_{P_3}$ for $4$ colors.}
\label{fig:bigMatrix}
\end{figure}
In Section \ref{sec:TMMEhrhart}, we establish a \emph{compactified} transfer matrix $L$ whose size does \emph{not} depend on $k$ and which can be used to count the number of proper $k$-colorings of $G \times P_n$ for \emph{all} $k$. Furthermore, we show that the biggest eigenvalue of $L$ equals the biggest eigenvalue of $A_{M_G}$ for all $k$. We also give a combinatorial and a geometric interpretation for the entries of $L$.

\subsection{Ehrhart Theory and Inside-Out Polytopes}
\label{sec:InsideOut}
Throughout this section, let $n$ be fixed. This implies that the size of the graphs $G \times P_n$ and $G \times C_n$ is fixed, too. Under this assumption, Problem \ref{prob:ColoringsPnandCn} reduces to computing the chromatic polynomial of a given graph. We will use the perspective of inside-out polytopes and Ehrhart theory developed by Beck and Zaslavsky \cite{BeckZaslavsky06} to understand chromatic polynomials. For a nice introduction to Ehrhart theory, we refer the interested reader to \cite{CCD}.

The fundamental geometric objects in this section are polytopes. A \emph{(convex) polytope $P\subset \mathbb{R}^d$} is the convex hull of finitely many points $a_1$, $a_2$ ,$\dots$, $a_s\in \mathbb{R}^d$, and we write
\[
P = \operatorname{conv} \{a_1, \dots, a_s \}: =\left \{ \sum_{i =1}^s \lambda_i a_i \colon \, \sum_{i=1}^s \lambda_i =1, \lambda_i \geq 0  \text{ for all } i \in [s]\right\}.
\]
The inclusion-minimal set $\{v_1,v_2,\dots,v_r \}$ such that $P = \operatorname{conv} \{v_1,v_2,\dots,v_r \}$ is called the \emph{vertex set} and its elements are called the \emph{vertices} of $P$. If the vertex set is contained in $\mathbb{Z}^d$, we say that $P$ is a \emph{lattice polytope}.

If $d$ is the dimension of the affine hull of $P$, we call $P$ a $d$-polytope. Polytopes form an intruiging subject in itself, and we refer the interested reader to \cite{Ziegler}. For positive integers $t$, we define the counting function
\[
E_{P}(t) : = \# \left(tP \cap \mathbb{Z}^d \right).
\]
The function $E_{P}(t)$ is called the \emph{Ehrhart function} of $P$. Ehrhart \cite{Ehrhart} famously proved that $E_P$ is a quasipolynomial if the vertices of $P$ are rational and it is even a polynomial if the vertices are integer points. The leading coefficient of the Ehrhart (quasi-)polynomial is given by the \emph{volume} of the polytope and if $P$ is a lattice polytope, the constant coefficient equals 1. The following result was conjectured by Ehrhart, but first proven in full generality by Macdonald \cite[Prop. 4.1]{Macdonald}.
\begin{theorem}[Ehrhart-Macdonald reciprocity]
Let $P$ be a $d$-dimensional rational polytope. Then
\begin{equation}
\label{eq:Ehrhart-Macdonald reciprocity}
E_{P}(-t) = (-1)^d E_{P^{\circ}}(t),
\end{equation} 
where $E_{P^{\circ}}(t)$ counts the number of integer points in the interior of $tP$.
\end{theorem}

The following definitions are taken from $\cite{BeckZaslavsky06}$. A \emph{hyperplane arrangement $\mathcal{H}$} is a set of finitely many linear or affine hyperplanes in $\mathbb{R}^d$. An \emph{open region }is a connected component of $\mathbb{R}^d \setminus \bigcup_{H \in \mathcal{H}} H$. A \emph{closed region} is the topological closure of an open region. Moreover, we define the \emph{intersection semilattice}
\[
\mathcal{L}(\mathcal{H}) : = \left\{ \bigcap \mathcal{S} \colon \mathcal{S} \subset \mathcal{H} \text{, and} \bigcap \mathcal{S} \neq \emptyset\right\},
\]
where the order is given by reverse inclusion. The minimal element is $\hat 0 = \mathbb{R}^d$. The elements of $\mathcal{L}(\mathcal{H})$ are sometimes called \emph{flats}. $\mathcal{L}(\mathcal{H})$ is therefore a partially ordered set --- or poset for short ---  and we recursively define the \emph{M\"obius function $\mu\colon \mathcal{L}(\mathcal{H}) \times \mathcal{L}(\mathcal{H}) \rightarrow \mathbb{Z}$} by
\[
\mu(r,s): = 
\begin{cases}
0 &\text{if } r\nleq s,\\
1 &\text{if } r = s,\\
-\sum_{r \leq u <s }\mu(r,u) &\text{if } r<s.\\
\end{cases}
\]
The \emph{characteristic polynomial $p_{\mathcal{H}}$} of a hyperplane arrangement $\mathcal{H}$ is defined as
\[
p_{\mathcal{H}}(\lambda) : = 
\begin{cases}
0 &\text{if }\mathcal{H} \text{ contains the degenerate hyperplane }\mathbb{R}^d,\\
\sum_{s \in \mathcal{L}(\mathcal{H})} \mu(\hat 0, s) \lambda^s &\text{otherwise}.\\
\end{cases}
\]
June Huh has shown that the coefficients of $p_{\mathcal{H}}$ are alternating in sign and the absolute value of the coefficients form a log-concave sequence, see \cite[Thm. 3]{Huh}.

Let $P\subset \mathbb{R}^d$ be a rational, closed, convex $d$-polytope and let $\mathcal{H}$ be an arrangement of \emph{rational} hyperplanes that meets $P$ transversally, i. e., every flat 
\[
u \in \left\{\bigcap S \colon \, S \subset \mathcal{H} \text{ and } \bigcap S \neq \emptyset \right\}
\]
that intersects the topological closure of $P$ also intersects the interior $P^{\circ}$. Here rational means that all vertices of $P$ lie in $\mathbb{Q}^d$ and all hyperplanes in $\mathcal{H}$ are specified by equations with rational coefficients. Following \cite{BeckZaslavsky06}, we call $(P,\mathcal{H})$ a \emph{rational inside-out polytope}. A \emph{region} of $(P,\mathcal{H})$ is one of the components of $P \setminus \bigcup \mathcal{H}$ or the closure of one such component. If the vertex set is a subset of $\mathbb{Z}^d$, we call $(P,\mathcal{H})$ a \emph{lattice inside-out polytope}. The \emph{multiplicity of $x \in \mathbb{R}^d$ with respect to $\mathcal{H}$ is}
\[
m_{\mathcal{H}}(x): = \# \text{closed regions of }\mathcal{H} \text{ containing }x. 
\]
The \emph{multiplicity with respect to $(P,\mathcal{H})$} is
\[
m_{P,\mathcal{H}}(x) : = 
\begin{cases}
\# \text{closed regions of } (P,\mathcal{H}) \text{ that contain }x, &\text{ if }x\in P,\\
0 & \text{otherwise.}\\
\end{cases}
\]
We define the \emph{closed Ehrhart quasipolynomial}
\[
E_{P,\mathcal{H}}(t) : = \sum_{x \in \mathbb{Z}^d} m_{tP,\mathcal{H}} (x),
\] 
where $t\in \mathbb{Z}_{\geq 1}$ and $tP$ denotes the $t^{\text{th}}$ dilate of $P$. Similarly, we define the \emph{open Ehrhart quasipolynomial}
\[
E^{\circ}_{P,\mathcal{H}}(t) := \# \left( \mathbb{Z}^d \cap t \left[P \setminus \bigcup H \right] \right)\text{.}
\]
With the notation from above, we have:
\begin{theorem}[{\cite[Thm. 3.1]{BeckZaslavsky06}}]
\label{thm:Moebius}
Let $P\subset \mathbb{R}^d$ be a $d$-dimensional polytope and let $\mathcal{H}$ be a hyperplane arrangement not containing the degenerate hyperplane $\mathbb{R}^d$. Then
\begin{equation}
\label{eq:openmoebius}
E^{\circ}_{P,\mathcal{H}}(t) = \sum_{u \in \mathcal{L}(\mathcal{H})} \mu(\hat 0,u) \# \left(\mathbb{Z}^d \cap t P \cap u \right)
\end{equation}
and if $\mathcal{H}$ is transverse to $P$, we have
\begin{equation}
\label{eq:closedmoebius}
E_{P,\mathcal{H}}(t)=\sum_{u \in \mathcal{L}(\mathcal{H})} \left|\mu(\hat 0,u)\right| \# \left(\mathbb{Z}^d \cap t P \cap u \right),
\end{equation}
where $\mu$ is the M\"obius function of $\mathcal{L}(\mathcal{H})$.
\end{theorem}
Since we assume $P$ to be full-dimensional, the hyperplane arrangement subdivides $P$ into closed regions $R_1$, $R_2$, $\dots$, $R_m$. Moreover, we have that
\begin{equation*}
E_{P,\mathcal{H}}(t)= \sum_{i=1}^m E_{R_i}(t) \qquad \text{and} \qquad E^{\circ}_{P^{\circ},\mathcal{H}}(t):= \sum_{i=1}^m E_{R^{\circ}_i}(t) \text{,}
\end{equation*}
where $E_{R_i}(t)$ is the classical Ehrhart quasipolynomial of the closed region $R_i$, and the interior of $R_i$ is with respect to the topology of the ambient space $\mathbb{R}^d$, see \cite[(4.2)]{BeckZaslavsky06}. We remark that $E^{\circ}_{P^{\circ},\mathcal{H}}(t)$ does not count any integer point on the facets of $P$, whereas $E^{\circ}_{P,\mathcal{H}}(t)$ also counts integer points on facets. Furthermore, there is a reciprocity result:
\begin{theorem}[{\cite[Theorem 4.1]{BeckZaslavsky06}}]
If $(P,\mathcal{H})$ is a $d$-dimensional lattice inside-out polytope, then $E_{P,\mathcal{H}}(t)$ and $E^{\circ}_{P^{\circ},\mathcal{H}}(t)$ are polynomials of degree $d$ and they also satisfy the reciprocity theorem
\begin{equation}
\label{eq:reciprocity}
E^{\circ}_{P^{\circ},\mathcal{H}}(t) = (-1)^dE_{P,\mathcal{H}}(-t)\text{.}
\end{equation}
\end{theorem}
In Section \ref{sec:EhrhartAndSymmetry}, we will intersect inside-out polytopes with hyperplanes, so we will need the following:
\begin{corollary}[{\cite[Corollary 4.3]{BeckZaslavsky06}}]
\label{cor:LowerDimReci}
Let $D$ be a discrete lattice in $\mathbb{R}^d$, let $P$ be a $D$-fractional convex polytope, i. e., all vertices of $P$ lie in $t^{-1}D$ for some $t \in \mathbb{Z}_{>0}$, and let $\mathcal{H}$ be a hyperplane arrangement in $s:= \operatorname{aff}(P)$ that does not contain the degenerate hyperplane. Then $E_{P, \mathcal{H}} (t)$ and $E^{\circ}_{P^{\circ}, \mathcal{H}} (t)$ are quasipolynomials in $t$ that satisfy the reciprocity law
\[
E^{\circ}_{P^{\circ}, \mathcal{H}} (t) = (-1)^{\operatorname{dim } s} E_{P, \mathcal{H}} (-t).
\]
\end{corollary}
The next result connects inside-out polytopes to proper colorings of graphs.
\begin{theorem}[{\cite[Theorem 5.1]{BeckZaslavsky06}}]
\label{thm:ChromEhrhart}
Let $G$ be an ordinary graph on $n$ vertices and let $P=[0,1]^n$. Moreover, we define
\[
\mathcal{H}(G) : = \left\{ x_i = x_j \colon \left\{x_i, x_j\right\} \in E \right\}\text{.} 
\]
Then
\begin{equation}
E^{\circ}_{P^{\circ},\mathcal{H(}G)}(t) = (-1)^n E_{P,\mathcal{H}(G)}(-t) =  \chi_G(t-1)\text{.}
\end{equation}
\end{theorem}
 
\begin{remark}
The intuition behind this theorem is the following: Every integer point in the interior of $tP$ corresponds to a (not necessarily proper) coloring of $G$. If $G$ contains the edge $\{x_i , x_j \}$, then any proper coloring cannot have an integer point on the hyperplane $x_i = x_j$. Therefore, every integer point in $t \left(P^{\circ}\setminus\bigcup_{H \in \mathcal{H}(G)} H\right)$ corresponds to a proper coloring of $G$ and vice versa. 
\end{remark}
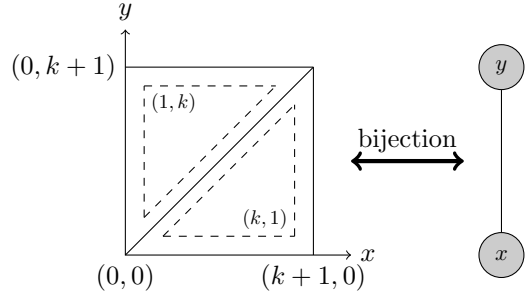
\begin{figure}[H]
\center
\begin{tikzpicture}[scale=.5]

\draw [<->] (0,6) -- (0,0) -- (6,0);
\node[above] at (0,6) {$y$};
\node[right] at (6,0) {$x$};
\draw (0,0) -- (5,5);
\draw (0,5) -- (5,5);
\draw (5,0) -- (5,5);
\draw[dashed](1,0.5) --(4.5,.5);
\draw[dashed] (4.5,.5)--(4.5,4);
\draw[dashed](1,.5)--(4.5,4);
\draw[dashed](0.5,1) --(.5,4.5);
\draw[dashed] (.5,4.5)--(4,4.5);
\draw[dashed](.5,1)--(4,4.5);

\node[below] at (0,0) {$(0,0)$};
\node[below] at (5,0) {$(k+1,0)$};
\node[left] at (0,5) {$(0,k+1)$};
\node[above left,scale=.75] at (4.5,0.5) {$(k,1)$};
\node[below right,scale=.75] at (.5,4.5) {$(1,k)$};

\draw [ultra thick,<->] (6,2.5) -- (9,2.5);
\node[above] at (7.5,2.5) {bijection};
\node[circle,draw=black,fill=white!80!black,minimum size=20,scale = .85]
(1) at (10,0) {$x$};
\node[circle,draw=black,fill=white!80!black,minimum size=20,scale = .85]
(2) at (10,5) {$y$} ;
\draw (1) -- (2);
\end{tikzpicture}
\caption{Integer points in dashed triangles correspond to proper $k$-colorings of $P_2$.}
\label{fig:Inside-Out}
\end{figure}
One immediate consequence of Theorem \ref{thm:ChromEhrhart} is that $\chi_G$ is a polynomial of degree $n$ and the leading coefficient is $1$, as the sum of the volumes of the regions of the subdivided cube equals $1$.

\section{Transfer-Matrix Methods Meet Ehrhart Theory}
\label{sec:TMMEhrhart}
\subsection{Enter Symmetry}
\label{sec:Symmetry}
In this section, we combine transfer-matrix methods with Ehrhart theory. We will use group actions and orbits to introduce a \emph{compactified transfer matrix $L$}, whose size does not depend on $k$. There are two types of symmetries that appear. Firstly, the set of proper $k$-colorings in Problem \ref{prob:ColoringsPnandCn} stays invariant under a permutation of colors. Simply renaming the colors does not change the graph $M_G$. For example, in Figure \ref{fig:TMM Graph}, the vertex $(1,2,1)$ has as many neighbors as the vertex $(3,2,3)$. Secondly, the graph $G$ itself also has a symmetry group, called the \emph{automorphism group of $G$}. We first quotient by the group coming from permuting the colors and then we quotient out by a possibly trivial subgroup of the automorphism group. 
\begin{definition}
\label{def:orbits}
Let $G$ be a simple graph on $N$ vertices and let $\mathcal{C}$ be the set of proper $k$-colorings, where $k\geq N$. Let $\mathfrak{S}_k$ be the symmetric group on $k$ elements and let $\mathcal{G}$ be a possibly trivial subgroup of the automorphism group of $G$. The group $\mathfrak{S}_k$ acts on $\mathcal{C}$ by permuting the colors and it gives rise to orbits $\tilde{o}_1, \dots, \tilde{o}_q$. The group $\mathcal{G}$ is acting on $\tilde{o}_1, \dots, \tilde{o}_q$ giving rise to orbits $o_1, \dots, o_p$. 
\end{definition}

\begin{example}
\label{ex:hiddensymmetrypreparation}
Let $G = C_5$.  We first quotient by permutations of colors. This group action induces orbits $\tilde{o}_1$, $\tilde{o}_2$, $\dots$, $\tilde{o}_{11}$ (as computer generated in no particular order):
\[
\begin{array}{c|c|c|c|c|c}
14\,\,2\,\,3\,\,5  &   14\,\,2\,\,35 & 14\,\,25\,\,3 & 1\,\,24\,\,3\,\,5 & 1\,\,24\,\,35 & 1\,\,2\,\,3\,\,4\,\,5  \\ 
\hline
 1\,\,2\,\,35\,\,4 & 1\,\,25\,\,3\,\,4 & 13\,\, 24 \,\, 5 & 13 \,\, 2 \,\, 4 \,\, 5 & 13 \,\, 25 \,\, 4 \\
\end{array}
\]
The automorphism group of $C_5$ is the dihedral group generated by $(12345)$ and $(1)(25)(34).$ The 11 partitions of the vertex set end up in 3 orbits $o_1$,$o_2$, and $o_3$ after quotienting by the dihedral group. The classes are represented by:
\[
\begin{array}{c|c|c}
1\,\,2\,\,3\,\,4\,\,5 & 1\,\,24\,\,35 & 1\,\,2\,\,4\,\,35 .\\
\end{array}
\]

\end{example}
\begin{remark}
For $k\geq N$, quotienting by $\mathfrak{S}_k$ always gives the same number of orbits. This is due to the fact that every orbit $\tilde{o}$ can be seen as a partition of the vertex set into independent sets. In particular, the number of orbits is \emph{finite}. Therefore, we also get that the number of orbits $o_1, \dots, o_p$ is the same for all $k \geq N$. In Definition \ref{def:orbits}, we only assume that $\mathcal{G}$ is a subgroup of the automorphism group, as it is in general difficult to determine the full automorphism group. 
\end{remark}
This enables us to define a matrix $L$ encoding the necessary combinatorial information, whose size is independent of the number of colors $k$. Following the arguments described in \cite{Ciucu}, we give the following definition:
\begin{definition}
\label{def:DefinitionL}
Let $G$ be a graph on $N$ nodes. Let $o_1$, $o_2$, $\dots$, $o_p$ be orbits as defined in Definition \ref{def:orbits} and let $k\geq N$ be any integer. Let $A_{M_G}$ be the transfer matrix of the graph $M_G$. We define a $p \times p$ matrix $L$ whose entries are given by
\[
 L_{i,j}= L_{o_i,o_j}= \sum_{m} a_{i,m},
\]
where the sum ranges over all elements in orbit $o_j$ and where $i$ is \emph{any} row of $A_{M_G}$ that corresponds to a representative of orbit $o_i$. $L$ is called the \emph{compactified transfer matrix of $G$.}
\end{definition}
We want to illustrate this definition for $k=4$. The general case will be treated in Example \ref{ex:CompactifiedTMM2}.
\begin{example}
\label{ex:CompactifiedTMM1}
We illustrate this procedure in the following example, where $G= P_3$. The orbits are $o_1 = \{\{1,3\},\{2\} \}$ and $o_2 = \{\{1\},\{2\},\{3\}\}$, so we can expect a $2 \times 2$ matrix $L$. In this matrix, the $(i,j)$-entry counts the number of colorings where the first $P_3$ is colored by a fixed representative of $o_i$ and the second $P_3$ is colored by any element in $o_j$.

Figure \ref{fig:bigMatrix} shows the transfer matrix for 4-colorings of $P_3$. By following Definition \ref{def:DefinitionL} for $k=4$, we can quotient out the orbits and we obtain the matrix
\[
L =\left(
\begin{array}{cc}
7 & 10 \\
5 & 11 \\
\end{array}\right),
\] 
which has the same maximal eigenvalue as the original transfer matrix in Figure \ref{fig:bigMatrix}, see Corollary \ref{cor:sameeigenvalue}. The entries of $L$ equal the row sums within each orbit/rectangle of the matrix in Figure \ref{fig:bigMatrix}.
\end{example}

\begin{remark}
\label{rem:SizeOfL}
It also makes sense to define $L$ for $1\leq k <N$. However, the size of $L$ will be smaller as not all orbits appear. For instance, the orbit, where all colors are different, cannot appear if $k<N$. As we will see in Corollary \ref{cor:ChromPolyProduct}, it makes sense to talk about $L$ even for $k<N$  provided we multiply $L$ by appropriate row weights that are $0$ if $k$ is too small for an orbit to appear.
\end{remark}
Since the entries of $L$ are nonnegative integers, the Perron-Frobenius theorem ensures that the biggest eigenvalue is real and positive. Furthermore, the following lemma implies that the biggest eigenvalue of $A_{M_G}$ and the biggest eigenvalue of $L$ agree.

\begin{lemma}[\cite{Ciucu}, Lemma 3.2]
\label{lem:Ciucu}
If $N$ is a nonnegative matrix that commutes with a group of permutation matrices $G$, then the largest eigenvalue of $N$ is the same as the largest eigenvalue of $N$ acting on the subspace of $G$-invariants.
\end{lemma}
\begin{corollary}
\label{cor:sameeigenvalue}
With the notation from above, the biggest eigenvalue of $L$ and the biggest eigenvalue of $A_{M_G}$ agree.
\end{corollary}
\begin{remark}
\label{rem:CombInterpretation}
There is also a combinatorial reformulation of Definition \ref{def:DefinitionL}. $L_{i,j}$ counts the number of colorings of $G \times P_2$, where $G \times \{1\}$ is colored by a fixed representative of $o_i$ and the coloring of $G \times \{ 2\}$ is in $o_j$.
\end{remark}
In Section \ref{sec:EhrhartAndSymmetry}, we will see that the entries of $L$ are indeed polynomials in $k$. We also give a geometric interpretation of the entries.

\subsection{Ehrhart Theory and Symmetry}
\label{sec:EhrhartAndSymmetry}
As we have seen in Remark \ref{rem:CombInterpretation}, the entries of $L$ can be interpreted as counting proper $k$-colorings where one part of the graph is fixed by a coloring and another part has to lie in a given orbit. In this section, we will use inside-out polytopes to show that this counting function is indeed a polynomial. However, we first start with a small result concerning graph colorings that lie in a given orbit.

Let $G= (V,E)$ be a finite, simple graph and let $\tilde{o}$ be an orbit as defined in Definition \ref{def:orbits}, i. e., we do \emph{not} quotient out by graph automorphisms. The graph $G$ determines a hyperplane arrangement
\[
\mathcal H(G) = \{x_i = x_j \colon \{i,j\} \in E \}.
\]

 We define the $\tilde{o}$-restricted chromatic polynomial
\[
\chi_{(G, \tilde{o})} (t) = \# \text{number of proper }t\text{-colorings of }G \text{ lying in orbit } \tilde{o}.
\]

As mentioned above, every orbit can be described by a partition of the vertex set $V$ into independent sets. Vertices in the same independent set are colored by the same color. Hence, for every independent set $I\in \tilde{o} $, we get an additional hyperplane arrangement 
\[
\mathcal{H} _I:= \left\{ x_i = x_j \colon i, \, j \in I \right\}.
\] 
Moreover, we also get a hyperplane arrangement
\[
\mathcal{H}_{\tilde{o},I}:= \left\{ x_i = x_j \colon i\in I \text{ and } j\notin I \right\}.
\] 
of forbidden hyperplanes by requiring that elements in different independent sets have different colors. Lastly, we define the hyperplane arrangement
\[
\mathcal H(G, \tilde o) = \mathcal{H}(G) \cup \bigcup_{I \in \tilde{o}} \mathcal{H}_{\tilde{o},I}.
\]

With this set-up, we now have:
\begin{theorem}
\label{thm:RestrictedChromaticEhrhart}

Let $\mathcal{H} _I$, $\mathcal H(G, \tilde o)$, and $\mathcal{H}_{\tilde{o},I}$ be defined as above. Moreover, let 
\[
P_{\tilde{o}} : = [0,1]^{n}\cap \left(\bigcap_{I\in \tilde{o}} H_I \right).
\]
Then
\begin{equation}
\label{eq:RestrictedEhrhart}
(-1)^{s} E_{P_{\tilde{o}}, \mathcal{H}(G,\tilde{o})}(-t)=E^{\circ}_{\operatorname{relint}(P_{\tilde{o}}), \mathcal{H}(G,\tilde{o})}(t) = \chi_{G, \tilde{o}}(t-1),
\end{equation}
where $s$ is the number of colors used in orbit $\tilde{o}$. Furthermore, $\chi_{G, \tilde{o}}$ is a polynomial of degree $s$ with leading coefficient $1$.
\end{theorem}
\begin{proof}
We first note that $P_{\tilde{o}}$ is an $s$-dimensional lattice inside-out polytope with volume $1$ (seen as an $s$-dimensional polytope). The hyperplane arrangement $\mathcal{H}(G, \tilde{o})$ subdivides $P_{\tilde{o}}$ into regions $R_1$,  $R_2,$  $\dots$, $R_m$, where $\operatorname{dim } R_i = \operatorname{dim }P_{\tilde{o}}$. We know that $E^{\circ}_{\operatorname{relint}(P_{\tilde{o}}), \mathcal{H}(G,\tilde{o})}(t)$ counts the integer points in the relative interior of the $t^{\text{th}}$-dilate of these regions. The integer points in these regions correspond to proper $(t-1)$-colorings of $G$ that lie in $\tilde{o}$, and every proper $(t-1)$-coloring in $\tilde{o}$ corresponds to an integer point in a $t \operatorname{relint} R_i$ for some $i$, which proves (\ref{eq:RestrictedEhrhart}). The claim about the degree of $\chi_{G,\tilde{o}}$ follows from the dimension and volume of $P_{\tilde{o}}$.
\end{proof}
\begin{remark}
The same result could also be obtained by forming a quotient graph where all vertices in the independent sets of $\tilde{o}$ get identified. Now counting proper $k$-colorings of this quotient graph is the same as counting proper $k$-colorings that lie in $\tilde{o}$. This however does not directly work if the also quotient out by graph automorphisms.
\end{remark}
\begin{remark}
The statement and the proof still hold for an orbit $o$ if we additionally quotient out by graph automorphisms, except that the leading coefficient will be the number of orbits $\tilde{o}$ that are in the preimage of orbit $o$.
\end{remark}
We now want to apply this geometric setting to the transfer-matrix theory to interpret the entries of $L$ in terms of Ehrhart polynomials. Recall that we want to color $\Gamma :=G\times P_2$, where $G$ is a graph on $N$ nodes. In the classical setting, we have a subdivision of the $2N$-dimensional unit cube stemming from the edges of the graph $\Gamma$. However, we want to further refine this subdivision according to the orbit structure. As we will see, this subdivision nicely resembles the symmetry of the orbits.
 
Let $G$ be a graph with vertices $\{v_1,v_2,\dots,v_N \}$, let $o_i, o_j \in O_G$ be defined as in Definition \ref{def:orbits}, and let $V(P_2) = \{1,2\}$. Pick a representative $c$ of $o_i$ such that $c(v_i)\leq N$ and color the first $N$ vertices accordingly. This defines a $(G \times \{1\})$-restricted coloring of $G\times P_2$, which we will call an $o_i$-restricted coloring. Recall that $L_{o_i,o_j}$ counts the number of $o_i$-restricted colorings such that the coloring of $G \times \{2\}$ is an element of $o_j$. 

Since the entries will correspond to restricted colorings, we first state a general result about restricted colorings. It is quite possible that this result is already known and it has a similar flavor as \cite[Thm. 4.1]{BeckZaslavsky06}.
\begin{theorem}
\label{thm:restrictedReciprocity}
Let $\Gamma = (\{1,2,\dots,n \},E)$ be a graph and fix a proper $k'$-coloring $c' \colon V' \rightarrow \{1,2,\dots,k' \}$ on the induced subgraph $\left. G\right|_{V'}$ for a subset $V' \subset V(\Gamma)$. Then, for $k \geq k'$, the restricted chromatic polynomial 
\begin{equation}
\label{eq:RestrictedChromaticPolynomial}
\chi_{c', \Gamma}(k)= \# \text{proper $k$-colorings c of }\Gamma \text{ such that }\left.c \right|_{V'} = c'.
\end{equation} is a polynomial of degree $\#V - \#V'=:s$ with leading coefficient $1$, whose coefficients $a_i$ alternate in sign, and whose absolute values of the coefficient form a \emph{log-concave sequence}, i.e., $a_i^2 \geq a_{i-1}a_{i+1}$ holds for $0 < i<s$. The second highest coefficient $a_{n-1}$ is given by
\[
-a_{n-1} =\# \text{edges }\{ v_i, v_j\} \text{ such that } \{v_i, v_j \}\nsubseteq V'.
\]
Moreover, we have the reciprocity statement
\begin{align*}
\label{eq:RestrictedReciprocity}
\chi_{c', \Gamma}(-k) &= (-1)^{s} \# (\alpha, c)\text{ of }\Gamma \text{ with } \left.c \right|_{V'}=c'\\
					 &= (-1)^s \chi_{c',\Gamma}(k),\\
\end{align*}
where $(\alpha,c)$ is a pair of an acyclic orientation $\alpha$ and a compatible (\emph{not necessarily proper}) $k$-coloring, $c$, of $\Gamma$. 
\end{theorem}

\begin{proof}
Let $\mathcal{H} = \mathcal{H}(\Gamma)$ be the hyperplane arrangement coming from the edges of the graph $\Gamma$. Let $P = [0,1]^n$. We intersect $(P,\mathcal{H})$ with the hyperplanes coming from the coloring $c'$, i.e., we intersect $(P,\mathcal{H})$ with the hyperplanes $x_i = c'(v_i)=:c'_i$ for all $v_i \in V'$. This induces a new inside-out polytope $(\overline{P},\overline{\mathcal{H}})$ of dimension $s$. This is illustrated in Figure \ref{fig:InducedInsideOutPolytope}. 
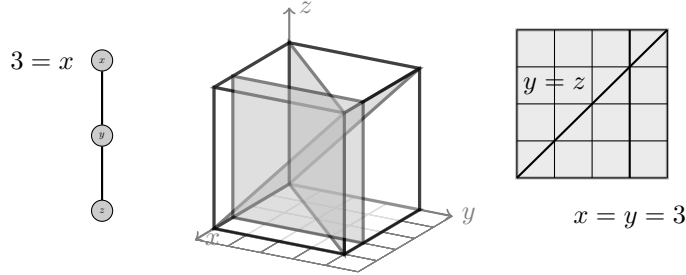
\begin{figure}[h]
\centering
\begin{subfigure}{.3\textwidth}
\centering
\begin{tikzpicture}[scale =.5]
\node[circle,draw=black,fill=white!80!black,minimum size=20, scale=.4] (1) at (-1,4) { $x$};
\node[circle,draw=black,fill=white!80!black,minimum size=20, scale=.4] (2) at (-1,2) { $y$};
\node[circle,draw=black,fill=white!80!black,minimum size=20, scale=.4] (3) at (-1,0) { $z$};
\draw[thick] (1)--(2)--(3);
\node[left] at (-1.5,4) {$\color{black}3=x\color{black}$};

\end{tikzpicture}
\end{subfigure}%
\begin{subfigure}{.3\textwidth}
\centering
\tdplotsetmaincoords{70}{120}
\begin{tikzpicture}
	[tdplot_main_coords,
		grid/.style={very thin,gray},
		axis/.style={->,gray,thick},
		cube/.style={very thick,fill=white,opacity=.5},
		square/.style={opacity=.5,very thick,fill=lightgray},scale=.5]
	%draw a grid in the x-y plane

%%%%%%%%%	
	\foreach \x in {-0,1,...,5}
		\foreach \y in {-0,1,...,5}
		{
			\draw[grid] (\x,0) -- (\x,5);
			\draw[grid] (-0,\y) -- (5,\y);
		}			

	%draw the axes
	\draw[axis] (0,0,0) -- (5,0,0) node[anchor=west]{$x$};
	\draw[axis] (0,0,0) -- (0,5,0) node[anchor=west]{$y$};
	\draw[axis] (0,0,0) -- (0,0,5) node[anchor=west]{$z$};

	%draw the bottom of the cube
	\draw[cube] (0,0,0) -- (0,4,0) -- (4,4,0) -- (4,0,0) -- cycle;
	
	%draw the back-right of the cube
	\draw[cube] (0,0,0) -- (0,4,0) -- (0,4,4) -- (0,0,4) -- cycle;

	%draw the back-left of the cube
	\draw[cube] (0,0,0) -- (4,0,0) -- (4,0,4) -- (0,0,4) -- cycle;

	%draw the front-right of the cube
	\draw[cube] (4,0,0) -- (4,4,0) -- (4,4,4) -- (4,0,4) -- cycle;

	%draw the front-left of the cube
	\draw[cube] (0,4,0) -- (4,4,0) -- (4,4,4) -- (0,4,4) -- cycle;

	%draw the top of the cube
	\draw[cube] (0,0,4) -- (0,4,4) -- (4,4,4) -- (4,0,4) -- cycle;

%forbidden hyperplanes
	\draw[square] (0,0,0) -- (0,4,4) -- (4,4,4) -- (4,0,0) -- cycle;
	\draw[square] (0,0,0) -- (4,4,0) -- (4,4,4) -- (0,0,4) -- cycle;
	\draw[square] (3,0,0) -- (3,4,0) -- (3,4,4) -- (3,0,4) -- cycle;
	
\end{tikzpicture}
\end{subfigure}%
\begin{subfigure}{.3\textwidth}
\centering
\tdplotsetmaincoords{80}{90}
\begin{tikzpicture}
[tdplot_main_coords,
		grid/.style={very thin,gray},
		axis/.style={->,blue,thick},
		cube/.style={opacity=.3,very thick,fill=lightgray},scale = .5]
			
\draw[cube] (0,28,-3) -- (0,32,-3) -- (0,32,1) -- (0,28,1) -- cycle;
\draw[thick] (0,28,-3) -- (0,32,1);
\draw[thick] (0,31,-3) -- (0,31,1);
%grid
\draw[] (0,28,-3) -- (0,28,1); 
\draw[] (0,29,-3) -- (0,29,1);
\draw[] (0,30,-3) -- (0,30,1);
%\draw[] (0,31,-3) -- (0,31,1);
\draw[] (0,32,-3) -- (0,32,1);

\draw[] (0,28,-3) -- (0,32,-3);
\draw[] (0,28,-2) -- (0,32,-2);
\draw[] (0,28,-1) -- (0,32,-1);
\draw[] (0,28,0) -- (0,32,0);
\draw[] (0,28,1) -- (0,32,1);

\node[below] at (0,31,-3.5) {$x=y=3$};
\node[above] at (0,29,-1) {$y=z$};
\end{tikzpicture}
\end{subfigure}
\caption{$P_3$ with vertex $x$ colored by $3$, the corresponding inside-out polytope $(P,\mathcal{H})$, and the \emph{induced} inside-out polytope $(\overline{P},\overline{\mathcal{H}})$.}
\label{fig:InducedInsideOutPolytope}
\end{figure}
Using an affine, unimodular map we can assume that $(\overline{P},\overline{\mathcal{H}}) \subset \mathbb{R}^{s}$ is full-dimensional. The integer points in $[1,k]^s$ that are not in $\overline{\mathcal{H}}$ are counted by $\chi_{c', \Gamma}(k)$, where we assume that $k\geq \max_i c_i'$.

Thus, by (\ref{eq:openmoebius}) we have that 
\[
\chi_{c', \Gamma}(k) = \sum_{u \in \mathcal{L}(\overline{\mathcal{H}})} \mu(\hat 0 ,u) \# \left(\mathbb{Z}^{s} \cap u \cap k \overline{P} \right)= \sum_{u \in \mathcal{L}(\overline{\mathcal{H}})} \mu(\hat 0 ,u)  k^{\operatorname{dim} u}.
\]
Note that this is actually the \emph{characteristic polynomial}, denoted $p_{\overline{\mathcal{H}}}$, of the induced hyperplane arrangement $\overline{\mathcal{H}}$, since $ \# \left(\mathbb{Z}^{s} \cap u \cap k \overline{P} \right) = k^{\dim u}$. The claim about the log-concavity now follows by a result of June Huh, see \cite[Cor. 27]{Huh}. The statement about the second highest coefficient follows, since the only terms of dimension $s-1$ come from flats of dimension $s-1$, which are exactly the hyperplanes coming from edges $\{ v_i, v_j\}$ such that  $\{v_i, v_j \}\nsubseteq V'$. We remark that this polynomial is alternating in sign, again by \cite[Thm 3.1]{Huh}. , 
\[
\chi_{c', \Gamma}(-k) =  (-1)^{s} \sum_{u \in \mathcal{L}(\overline{\mathcal{H}})} \left| \mu(\hat 0 ,u)\right| k^{\dim u}.
\]
This --- using again (\ref{eq:closedmoebius}) --- is equivalent to
\[
(-1)^{s} \chi_{c', \Gamma}(-k) = E_{\overline{P},\overline{\mathcal{H}}} (k-1)= (-1)^s \chi_{c', \Gamma}(k) = \sum_{x \in \mathbb{Z}^s } m_{((k-1) \overline{P},\overline{\mathcal{H}})} (x).
\]
Similar to \cite[Proof of Cor. 5.5]{BeckZaslavsky06}, one can now observe that the right-hand side counts the number of compatible pairs $(\alpha, c)$, where $c$ is a ---not necessarily proper--- $c'$-restricted $k$-coloring and $\alpha$ is an acyclic orientation. Here we implicitly used that $k\geq \max_{i} c_i'$ while applying \cite[Thm 3.1]{BeckZaslavsky06}.
\end{proof}

The following theorem states some basic facts about $L_{o_i,o_j}$:
\begin{theorem}
\label{thm:BasicGeometricFacts}
With the notation from above and with $k\geq N$, we have:
\begin{enumerate}
\item Every entry $L_{o_i, o_j}(k)$ equals the sum of Ehrhart polynomials of lattice inside-out polytopes of dimension $\# \text{colors of }o_j$ and hence is a polynomial of degree $\# \text{colors of }o_j$,
\item $L_{o_i,o_j}(k)$ is independent of the choice of the representative, i.e., it is well-defined,
\item $\# o_i \cdot L_{o_i, o_j}(k) = \chi_{G_{o_i,o_j}}(k)$, where
\begin{align*}
\chi_{G_{o_i,o_j}}(k) = &\# \text{proper } k\text{-colorings}\colon \text{coloring of }G \times \{1\} \text{ is in }o_i \text{ and } \\ &\text{ coloring of }G \times \{2\} \text{ is in }o_j.\\
\end{align*}
\end{enumerate}
\end{theorem}
\begin{proof}

We first prove the statement for the case, where we only quotient out by permutations of colors. Then we show that this implies the statement for orbits $o$ when we also quotient out by a subgroup of the automorphism group.

Let $G$ be a graph with vertices $\{v_1,v_2,\dots,v_N \}$, let $P_2$ be the path graph on 2 vertices, and let
\[
\mathcal{H}' = \{x_i = x_j \colon \, \{i,j \}\in E(G\times P_2) \}.
\]
Moreover, let $C = [0,1]^{2N}$. Let $\tilde{o}_i, \tilde{o}_j$ be given orbits of $G$ after quotienting out by permutations of colors. Let $\mathcal{I}$ be the set of independent sets of $\tilde{o}_j$. As above, for every $I \in \mathcal{I}$, we get
\[
\mathcal{H} _I:= \left\{ x_i = x_j \colon i, \, j \in I \right\},
\] 
and we get a set of forbidden hyperplanes 
\[
\mathcal{H}_{\tilde{o}_j,I}:= \left\{ x_i = x_j \colon i\in I \text{ and } j\notin I \right\}.
\] 
Now let $P = C \cap \left( \bigcup_I  \mathcal{H} _I\right)$ and let 
\[
\mathcal{H} = \mathcal{H}'  \cup \bigcup_I  \mathcal{H}_{\tilde{o}_j,I}:= \left\{ x_i = x_j \colon i\in I \text{ and } j\notin I \right\}.
\]

Then  $(P, \mathcal{H})$ is a lattice inside-out polytope and its dimension is 
\[
N + \text{number of colors in }\tilde{o}_j = N + \# \mathcal{I}.
\] 

We remark that this is the same hyperplane arrangement that one would obtain from a quotient graph of $G \times P_2$. This quotient graph can be obtained in the following way:
\begin{enumerate}
\item All vertices of $G \times \{ 2\}$ that are in the same independent set of $\tilde{o}_j$ get identified,
\item and we turn all vertices in the image of $G \times \{ 2\}$ into a clique.
\end{enumerate}   
Therefore, one can  now apply Theorem \ref{thm:restrictedReciprocity} to see that the entries of $L_{\tilde{o}_i,\tilde{o}_j}$ are polynomials.

If we now furthermore quotient by graph symmetries, we by definition fix a row and add all entries that are in columns indexed by orbits that get mapped to $o_j$. Thus, the the entries of $(L_{o_i,o_j})$ are also polynomials whose leading coefficient is the number of orbits in the preimage of $o_j$.

Now let $c$ and $c'$ be colorings in the same orbit $\tilde{o}_i$. Then there is a bijection of the colors mapping $c$ to $c'$. This permutation gives rise to a bijection of the lattice points in $t\operatorname{relint}(R_i)$. Since $\chi_{G_{\tilde{o}_i,\tilde{o}_j}} (t)$ counts the number of proper $t$-colorings such that $G \times \{ 1\}$ is an element of orbit $\tilde{o}_i$ and $G \times \{ 2\}$ is an element of orbit $\tilde{o}_j$ and we therefore get
\[
\# \tilde{o}_i \cdot L_{\tilde{o}_i, \tilde{o}_j}(t) = \chi_{G_{\tilde{o}_i, \tilde{o}_j}}(t).
\]

The last statement follows, since we moded out by a graph symmetry.
\end{proof}

Now that we have seen that the entries of $L$ can be interpreted as Ehrhart polynomials of inside-out polytopes, we give an explicit example.
\begin{example}[Example \ref{ex:CompactifiedTMM1} continued]
\label{ex:CompactifiedTMM2}
Again let $G= P_3$. Recall that the orbits are $o_1 = \{\{1,3\},\{2\} \}$ and $o_2 = \{\{1\},\{2\},\{3\}\}$, so we can expect a $2 \times 2$ matrix $L$. 

We do the same calculation as in Example \ref{ex:CompactifiedTMM1}, but for $k$ colors. For the matrix $L$ the entries will be polynomial for $k\geq 3,$ but for lower $k$ the polynomials might not make sense, as for small $k$ there are not enough colors for every orbit. On the other hand, the polynomials are of degree at most three. By explicit computer calculations for $k=3,4,5,6$ we infer that the matrix is
\[
\left(
\begin{array}{cc}
k^2-3k+3  &  k^3-6k^2+13k-10 \\
k^2-4k+5  &  k^3-6k^2+14k-13 \\
\end{array}
\right)
\]
for $k\geq 3$. In comparison, the adjacency matrix of $M_G$ of proper $k$-colorings of $P_3$ has dimension $k(k-1)^2 \times k(k-1)^2$.
\end{example}
In Example \ref{ex:CompactifiedTMM2}, the graph $P_3$ was so small that there was no graph automorphism identifying two orbits $\tilde{o}$ and $\tilde{o}'$. We now illustrate how the matrices are further reduced in size when the automorphisms of the underlying graphs are also considered. We first consider $k$-colorings of $G \times P_n$ with $G=C_5.$ It should be noted that this question can be addressed with ad hoc methods adapted to this particular choice of $G,$ but our method is completely general. In our method, we do not assume anything about $G.$ 

\begin{example}[Example \ref{ex:hiddensymmetrypreparation} continued]
\label{ex:hiddensymmetry}
Let $G=C_5$ be the graph for which we want to calculate the chromatic polynomial of $G \times P_n.$ We label the edges of $C_5$ by $12, 23, 34, 45$ and $51.$ We first quotient by permutations of colors. The 11 partitions of the vertices into independent sets are (as computer generated in no particular order):
\[
\begin{array}{c|c|c|c|c|c}
14\,\,2\,\,3\,\,5  &   14\,\,2\,\,35 & 14\,\,25\,\,3 & 1\,\,24\,\,3\,\,5 & 1\,\,24\,\,35 & 1\,\,2\,\,3\,\,4\,\,5  \\ 
\hline
 1\,\,2\,\,35\,\,4 & 1\,\,25\,\,3\,\,4 & 13\,\, 24 \,\, 5 & 13 \,\, 2 \,\, 4 \,\, 5 & 13 \,\, 25 \,\, 4 \\
\end{array}
\]
Since every partition corresponds to an orbit $\tilde{o}$ as defined in Definition \ref{def:orbits}, we expect an $11 \times 11$-matrix:
\begin{figure}[H]
\includegraphics[width= \textwidth]{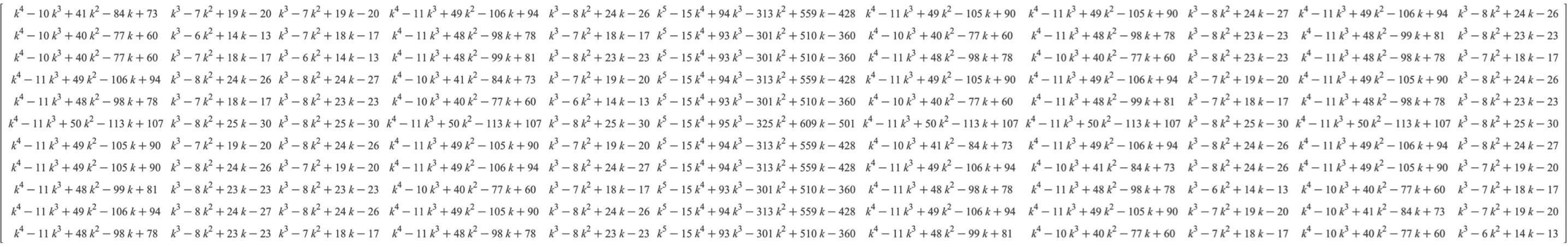}
\caption{$L$ matrix where we only quotient out by permutations of colors.}
\label{fig:eleven}
\end{figure}
The automorphism group of $C_5$ is the dihedral group generated by $(12345)$ and $(1)(25)(34).$ The 11 partitions of the vertex set end up in 3 orbits after quotienting by the dihedral group. The classes are represented by:
\[
\begin{array}{c|c|c}
1\,\,2\,\,3\,\,4\,\,5 & 1\,\,24\,\,35 & 1\,\,2\,\,4\,\,35 \\
\end{array}
\]
Adding up entries from columns --- indexed by orbits $\tilde{o}_i$ that get mapped to the same orbit $o$ --- of the $11 \times 11$-matrix gives an even more compactified version, the $3 \times 3$ matrix $L,$
\begin{tiny}
\[
\left(
\begin{array}{ccc}
k^5-15k^4+95k^3-325k^2+609k-501 &  5k^3-40k^2+125k-150 & 5k^4-55k^3+250k^2-565k+535 \\
k^5-15k^4+93k^3-301k^2+510k-360 &  5k^3-36k^2+96k-93     &  5k^4-53k^3+224k^2-449k+357 \\ 
k^5-15k^4+94k^3-313k^2+559k-428 & 5k^3-38k^2+110k-119   & 5k^4-54k^3+237k^2-506k+441 \\
\end{array}\right).
\]
\end{tiny}
\end{example}

\subsection{Main Results}
\label{sec:Main Results}
In this section, we show that $L$ defined as in Definition \ref{def:DefinitionL} behaves like a transfer matrix. We then deduce that the chromatic polynomial of $G \times P_n$ can be determined by computing powers of $L$. Moreover, we will use geometry to find an Ehrhart-theoretic interpretation for the entries of $L^n$.  This geometric interpretation allows us to deduce a reciprocity statement for the rows sums of $L^n$. We end this section by stating results about the biggest eigenvalue of $L$. 

Let $G$ be a graph let $o_1$, $o_2$, $\dots$, $o_p$ be as defined in Definition \ref{def:orbits}. Then $L = \left(L_{i,j}(k)\right)$ is a $p \times p$ matrix and the entry $(i,j) = (o_i,o_j)$  is given by 
\begin{align*}
L_{i,j} &= \# o_i \text{-restricted colorings of } G \times \{1\} \colon \text{coloring of }  G \times \{ 2\} \text{ lies in }o_j \\
	    &= \frac{\chi_{G_{o_i,o_j}}(k)}{\# o_i}.
\end{align*}

The next result shows that $L$ behaves like a transfer matrix.
\begin{theorem}
\label{thm:L^n_entry}
Let $V(P_{n+1}) = \{1,2,\dots, n+1\}$ and let $G$ be any graph. Let $o_1, o_2,\dots, o_p$ be the orbits as defined in Definition \ref{def:orbits}. Then, for $k \geq \# \text{colors used in }o_i$, the $(o_i,o_j)$-entry of $L^n$  counts the number of proper $k$-colorings of $G \times P_{n+1}$, where $G \times \{1\}$ is fixed by a coloring in orbit $o_i$, and where the coloring of $G \times \{n+1\}$ lies $o_j$. 
\end{theorem}
Before we prove Theorem \ref{thm:L^n_entry}, we illustrate the statement:
\begin{example}[Example \ref{ex:CompactifiedTMM1} continued]
Let $G = P_3$ with orbits $o_1 = \{\{1,3 \}, \{2\} \}$ and $o_2 = \{\{1\}, \{ 2\}, \{ 3\} \}$. Recall that 
\[L = 
\left(
\begin{array}{cc}
k^2-3k+3  &  k^3-6k^2+13k-10 \\
k^2-4k+5  &  k^3-6k^2+14k-13 \\
\end{array}
\right).
\]
Then $(L^{5})_{1,1}$ counts the number of colorings of $G \times P_6$, where the two black dots indicate that the two corresponding nodes need to be colored by the same color.
\begin{figure}[h]
\hspace{3.5cm}
\begin{tikzpicture}[darkstyle/.style={circle,draw,fill=gray!40,minimum size=20}, scale = .5]
\draw[step=1.5,thick] (0,-3) grid (7.5,0);
\foreach \x in {0,...,5}
{   \foreach \y in {0,...,2}
    {   \pgfmathtruncatemacro{\nodelabel}{\x+\y*5+1}
        \node[circle,draw=black,fill=white!80!black,minimum size=20, scale=.4] (\nodelabel) at (1.5*\x,-1.5*\y) { };
    }
}
\node[circle,fill=black,minimum size=20, scale=.4] at (7.5,0) {};
\node[circle,fill=black,minimum size=20, scale=.4] at (7.5,-3) {};
\node[left] at (-.5,0) {$1$};
\node[left] at (-.5,-1.5) {$2$};
\node[left] at (-.5,-3) {$1$};
\end{tikzpicture}
\caption{$G \times P_6$.}
\label{fig:TMMEx}
\end{figure}
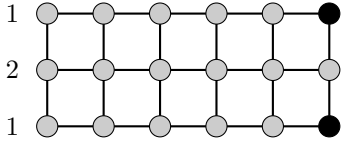
\end{example}

\begin{proof}[Proof of Theorem \ref{thm:L^n_entry}]
We induct on $n$. By construction, the statement is true for $n=1$, so let the statement be true for $G \times P_m$ for all $m\leq n$. Now  $V(P_{n+1})= \{1,2\dots,n+1 \}$. We denote the $(o_i,o_j)-$entry of $L^m \in \mathbb{Z}^{ p \times p}$ by $L^m_{i,j}$. Then
\[
L^{n}_{i,j} = (L^{n-1}L)_{i,j} = \sum_{k=1} ^p L^{n-1}_{i,k}L_{k,j}.
\]
By induction hypothesis, the entry $L^{n-1}_{i,k}$ counts the number of colorings where the the coloring of $G \times \{ 1\}$ is fixed by a representative in $o_i$ and the coloring of $G\times \{n\}$ lies in orbit $o_k$. Moreover, $L_{k,j}$ counts the colorings where the first $G$ is fixed by a representative in $o_k$ and the coloring of the second $G$ lies in $o_j$. Therefore, $L^{n-1}_{i,k}L_{k,j}$ counts the colorings where the coloring of $G \times \{1 \}$ is fixed by a representative of $o_i$, the coloring of $G\times \{n \}$ lies in $o_k$, and the coloring of $G \times \{ n+1\}$ lies in $o_j$. The sum is taken over all possible orbits and the claim follows.
\end{proof}
\begin{remark}
This shows that the entries of $L^n$ can be interpreted as sums of Ehrhart polynomials of (induced) inside-out polytopes assuming that the dilation factor is big enough. The inside-out polytopes can be explicitly described by following the construction from the proof of Proposition \ref{thm:BasicGeometricFacts}.
\end{remark} 
Moreover, this enables us to directly compute the chromatic polynomial of $G \times P_{n+1}$ from $L^n$.

\begin{corollary}
\label{cor:ChromPolyProduct}
Let $G \times P_{n+1}$ and $L$ be as above. Then
\begin{equation}
\label{eq:chromatic_poly_P_n}
\chi_{G \times P_{n+1}}(k) = (w_1(k),\dots, w_p(k)) L^n \mathbf{1},
\end{equation}
where $w_i (k)$ is the size of $o_i$ and $\mathbf{1}: = (1,\dots, 1)^t$.
\end{corollary}
\begin{proof}
Let $V(P_n) = \{1,2,\dots, n \}$. The $i^{\text{th}}$ entry $L_i$ of $ L^n \mathbf{1}$ counts the number of colorings where $G \times \{1\}$ is colored by a representative of $o_i$. By symmetry, the total number of colorings with the coloring of $G \times \{ 1\}$ being in $o_i$ equals
\[w_i(k) L_i \]
by Theorem~\ref{thm:BasicGeometricFacts}. Now $(w_1(k),\dots, w_p(k)) L^n \mathbf{1}$ sums over all possible orbits and the claim now directly follows.
\end{proof}
\begin{remark}
Even though the definition of $L$ implicitly assumes that the number of colors  $k \geq \# V(G)$, the corollary makes sense \emph{for all $k$}. If $k \leq \# V(G)$, then the weights $w_i$ of the orbits using more than $k$ colors are $0$.
\end{remark}
\begin{example}[Example \ref{ex:CompactifiedTMM1} continued]
The chromatic polynomial of $P_3 \times P_6$ is
\begin{equation*}
\chi(k) = (w_1, w_2) \left(
\begin{array}{cc}
k^2-3k+3  &  k^3-6k^2+13k-10 \\
k^2-4k+5  &  k^3-6k^2+14k-13 \\
\end{array}
\right)^{6-1} \begin{pmatrix}
1 \\ 1
\end{pmatrix},
\end{equation*}
where $w_1 = k(k-1)$ and $w_2 = k(k-1)(k-2)$.
\begin{figure}[H]\hspace{4 cm}
\begin{tikzpicture}[darkstyle/.style={circle,draw,fill=gray!40,minimum size=20}, scale = .5]
\draw[step=1.5,thick] (0,-3) grid (7.5,0);
\foreach \x in {0,...,5}
{   \foreach \y in {0,...,2}
    {   \pgfmathtruncatemacro{\nodelabel}{\x+\y*5+1}
        \node[circle,draw=black,fill=white!80!black,minimum size=20, scale=.4] (\nodelabel) at (1.5*\x,-1.5*\y) { };
    }
}
\end{tikzpicture}
\caption{$P_3 \times P_{ 6 }$.}
\end{figure}
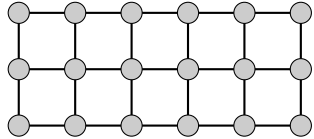

In general, the chromatic polynomial of $P_3 \times P_n$ equals
\begin{equation*}
\chi(k) = (w_1, w_2) \left(
\begin{array}{cc}
k^2-3k+3  &  k^3-6k^2+13k-10 \\
k^2-4k+5  &  k^3-6k^2+14k-13 \\
\end{array}
\right)^{n-1} \begin{pmatrix}
1 \\ 1
\end{pmatrix},
\end{equation*}
where $w_1 = k(k-1)$ and $w_2 = k(k-1)(k-2)$.
\end{example}

\begin{example}[Example \ref{ex:hiddensymmetry} continued]
Recall that $G = C_5$ and that the matrix $L$ was given by
\begin{tiny}
\[
\left(
\begin{array}{ccc}
k^5-15k^4+95k^3-325k^2+609k-501 &  5k^3-40k^2+125k-150 & 5k^4-55k^3+250k^2-565k+535 \\
k^5-15k^4+93k^3-301k^2+510k-360 &  5k^3-36k^2+96k-93     &  5k^4-53k^3+224k^2-449k+357 \\ 
k^5-15k^4+94k^3-313k^2+559k-428 & 5k^3-38k^2+110k-119   & 5k^4-54k^3+237k^2-506k+441 \\
\end{array}\right).
\]
\end{tiny}
Define a row vector with the size of each orbit, both counting the coloring and automorphism symmetries:
\[ v = ( k(k-1)(k-2)(k-3)(k-4), \,\,\,\, 5k(k-1)(k-2), \,\,\,\, 5k(k-1)(k-2)(k-3) )^T \]
Let $I$ be the $3 \times 3$ identity matrix and $\bf{1}$ the all-ones column vector of height 3. Then we have a nice formal generating function
\[ 
\begin{array}{rcl}
 \Xi_G(k,z) & = & \displaystyle \sum_{n=0}^\infty \chi_{G \times P_{n+1}}(k)z^n \\ 
 & = & \displaystyle \sum_{n=0}^\infty  vL^n {\bf{1}} z^n \\ 
 & = & \displaystyle v \left ( \sum_{n=0}^\infty  (zL)^n  \right) {\bf{1}}  \\ 
 & = & v(I-zL)^{-1}\bf{1}. 
 \end{array}
 \]
In general, if $L$ is an $m \times m$ matrix, we expect that
\[   \Xi_G(k,z) = 
\begin{array}{c}
\textrm{polynomial of $z$-degree $m-1$} \\
\hline
\textrm{polynomial of $z$-degree $m$} \\
\end{array}
\]
by calculating $(I-zL)^{-1}$ using cofactors. For some $G$ there is a mysterious cancellation and the $z$-degree of the denominator of $\Xi_G(k,z)$ is smaller than the size of the matrix. This is the case in our example, as
\[ \Xi_{C_5}(k,z) = k(k-1)(k-2) \frac{p_1(k)z+p_0(k)}{q_2(k)z^2+q_1(k)z+q_0(k)}  \]
where
\begin{small}
\[
\begin{array}{rcl}
p_0(k) & = & k^2-2k+2, \\
p_1(k) & = & -k^5+11k^4-44k^3+73k^2-42k+14, \\
q_0(k) & = & 1, \\
q_1(k) & = & -k^5+10k^4-46k^3+124k^2-198k+148, \\
 q_2(k) & = & k^8-19k^7+159k^6-767k^5+2339k^4-4627k^3+5800k^2-4212k+1362.
 \end{array}
 \]
\end{small}
\end{example}
In the previous example, the degree of the denominator of $\Xi_G(k,z)$ is smaller than expected, as if there was a symmetry waiting to be accounted for.
\begin{definition}
A graph $G$ has a \emph{hidden symmetry} if the denominator of 
\[ \Xi_G(k,z)  =  \sum_{n=0}^\infty \chi_{G \times P_{n+1}}(k)z^n \]
has a $z$--degree less than the order of
\[
\{ c : V(G) \rightarrow \mathbb{Z}_{\geq 1}\colon \textrm{$c$ is a proper coloring of $G$}  \}  /  \sim
\]
where $c \sim c'$ if $c=\alpha c' \beta$ for a bijection $\alpha$ of $\mathbb{N}$ and an automorphism $\beta$ of $G.$
\end{definition}

\begin{example}
We have done some computer calculations to tabulate graphs with hidden symmetries. The connected graphs with at most five vertices are in Figure~\ref{fig:hidden}.

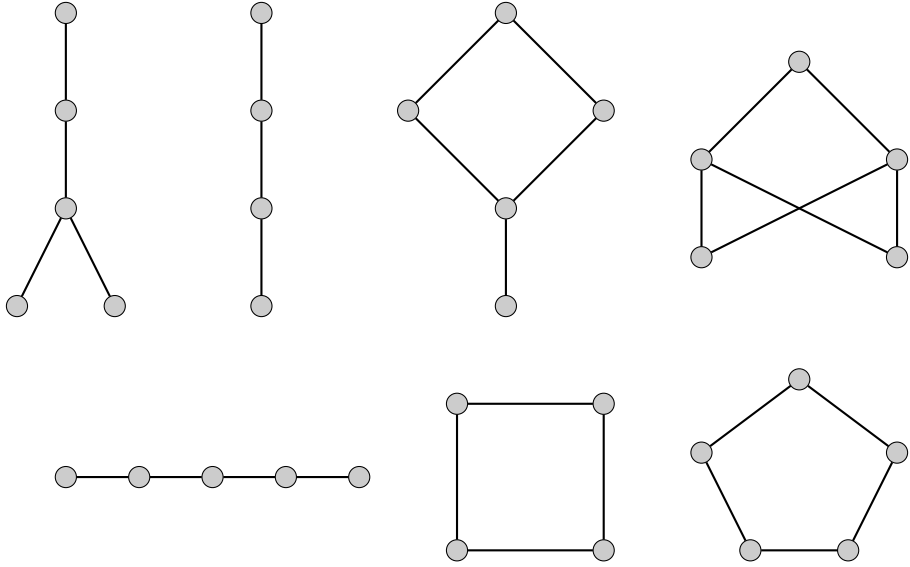
\begin{figure}[H]
\centering
\begin{tikzpicture}[darkstyle/.style={circle,draw,fill=gray!40,minimum size=20},scale= .65]
%P_5 move it
\node[circle,draw=black,fill=white!80!black,minimum size=20, scale=.4] (-1) at (0,-2+.5) {};
\node[circle,draw=black,fill=white!80!black,minimum size=20, scale=.4] (-2) at (1.5,-2+.5) {};
\node[circle,draw=black,fill=white!80!black,minimum size=20, scale=.4] (-3) at (3,-2+.5) {};
\node[circle,draw=black,fill=white!80!black,minimum size=20, scale=.4] (-4) at (4.5,-2+.5) {};
\node[circle,draw=black,fill=white!80!black,minimum size=20, scale=.4] (-5) at (6,-2+.5) {};
\draw [thick](-1) --(-2) -- (-3) -- (-4) -- (-5);

%P_4
\node[circle,draw=black,fill=white!80!black,minimum size=20, scale=.4] (-6) at (3+1,2) {};
\node[circle,draw=black,fill=white!80!black,minimum size=20, scale=.4] (-7) at (3+1,4) {};
\node[circle,draw=black,fill=white!80!black,minimum size=20, scale=.4] (-8) at (3+1,6) {};
\node[circle,draw=black,fill=white!80!black,minimum size=20, scale=.4] (-9) at (3+1,8) {};

\draw [thick](-6) --(-7) -- (-8) -- (-9);

%C_4

\node[circle,draw=black,fill=white!80!black,minimum size=20, scale=.4] (1) at (8,-1+1) {};
\node[circle,draw=black,fill=white!80!black,minimum size=20, scale=.4] (2) at (11,-1+1) {};
\node[circle,draw=black,fill=white!80!black,minimum size=20, scale=.4] (3) at (8,-4+1) {};
\node[circle,draw=black,fill=white!80!black,minimum size=20, scale=.4] (4) at (11,-4+1) {};

\draw [thick] (1) --(2) -- (4) -- (3) -- (1);

%C_5
\node[circle,draw=black,fill=white!80!black,minimum size=20, scale=.4] (5) at (14,-4+1) {};
\node[circle,draw=black,fill=white!80!black,minimum size=20, scale=.4] (6) at (16,-4+1) {};
\node[circle,draw=black,fill=white!80!black,minimum size=20, scale=.4] (7) at (17,-2+1) {};
\node[circle,draw=black,fill=white!80!black,minimum size=20, scale=.4] (8) at (15,-0.5+1) {};
\node[circle,draw=black,fill=white!80!black,minimum size=20, scale=.4] (9) at (13,-2+1) {};

\draw [thick](5) --(6) -- (7) -- (8) -- (9) -- (5);

%tree_1
\node[circle,draw=black,fill=white!80!black,minimum size=20, scale=.4] (10) at (-1,2) {};
\node[circle,draw=black,fill=white!80!black,minimum size=20, scale=.4] (11) at (1,2) {};
\node[circle,draw=black,fill=white!80!black,minimum size=20, scale=.4] (12) at (0,4) {};
\node[circle,draw=black,fill=white!80!black,minimum size=20, scale=.4] (13) at (0,6) {};
\node[circle,draw=black,fill=white!80!black,minimum size=20, scale=.4] (14) at (0,8) {};

\draw [thick](10) --(12) -- (13) -- (14); 
\draw [thick](11) --(12); 

%cube with leaf
\node[circle,draw=black,fill=white!80!black,minimum size=20, scale=.4] (15) at (8+1,2) {};
\node[circle,draw=black,fill=white!80!black,minimum size=20, scale=.4] (16) at (8+1,4) {};
\node[circle,draw=black,fill=white!80!black,minimum size=20, scale=.4] (17) at (6+1,6) {};
\node[circle,draw=black,fill=white!80!black,minimum size=20, scale=.4] (18) at (10+1,6) {};
\node[circle,draw=black,fill=white!80!black,minimum size=20, scale=.4] (19) at (8+1,8) {};

\draw [thick](15) --(16) -- (18) -- (19) --(17) -- (16);

%ugly graph
\node[circle,draw=black,fill=white!80!black,minimum size=20, scale=.4] (20) at (13,2+1) {};
\node[circle,draw=black,fill=white!80!black,minimum size=20, scale=.4] (21) at (13,4+1) {};
\node[circle,draw=black,fill=white!80!black,minimum size=20, scale=.4] (22) at (15,6+1) {};
\node[circle,draw=black,fill=white!80!black,minimum size=20, scale=.4] (23) at (17,4+1) {};
\node[circle,draw=black,fill=white!80!black,minimum size=20, scale=.4] (24) at (17,2+1) {};

\draw [thick](20) --(21) -- (22) -- (23) --(24) -- (21); 
\draw [thick](20) --(23); 
\end{tikzpicture}
\caption{The connected graphs on at most five vertices with a hidden symmetry. \label{fig:hidden}}
\end{figure}
One could speculate that hidden symmetry is something fairly trivial, since the common factor of the numerator and denominator when calculating  $\Xi_G(k,z)$ from cofactors is something straightforward. A piece of it usually seems be a power of $(z-1).$ But for example for $G=C_6$ the common factor is the non-trivial factor
$ (z-1)^3(k^7z^4-19k^6z^4+k^6z^3+147k^5z^4-20k^5z^3-598k^4z^4+157k^4z^3+1381k^3z^4-3k^4z^2-627k^3z^3-1821k^2z^4+37k^3z^2+1349k^2z^3+1289kz^4-173k^2z^2-1483kz^3-384z^4+3k^2z+364kz^2+659z^3-18kz-309z^2+35z-1).$ We conjecture that all paths and cycles on at least four vertices have a hidden symmetry.
\end{example}

\begin{theorem}
\label{thm:Reciprocity_L^n}
Let $L\in \mathbb{Z}^{p \times p}$ be as above, let $L^n_{i} : = \sum_{k=1}^p \left(L^n\right)_{i,k}$ be the $i^{\text{th}}$ row sum of $L^n$, and let $V(P_{n+1}) = \{1,2,\dots, n+1 \}$. Then, for $k \geq N=\# V(G)$, we have
\begin{equation}
\label{eq:reciprocity_L^n}
L^n_{i} (-k) =(-1)^{Nn} \# (\alpha, c)\text{ of } G \times P_{n+1} \text{ where } G \times \{ 1\} \text{ is colored by rep. of } o_i,
\end{equation}
where $(\alpha,c)$ is a pair of an acyclic orientation $\alpha$ and a compatible $o_i$-restricted $k$-coloring $c$.
\end{theorem}

\begin{proof}
$L^n_i (k)$ counts the number of colorings of $G \times P_{n+1}$, where $G \times \{ 1\}$ is fixed by a coloring $c'$. Now one can apply Theorem \ref{thm:restrictedReciprocity} and the claim follows.
\end{proof}

\begin{corollary}
\label{cor:Reciprocity}
Let $k\geq N$ Then
\begin{equation}
\label{eq:reciprocity_L}
\sum_{j = 1}^m L_{i , j}(-k) = (-1) \cdot \# (\alpha, c) \text{ where first } G \text{ is fixed by a representative of } o_i \text{,}
\end{equation}
where $(\alpha, c)$ is a pair of an acyclic orientation and compatible $k$-coloring.
\end{corollary}
\begin{example}[\ref{ex:CompactifiedTMM1} continued]
Figure \ref{fig:restricted reciprocity} illustrates Corollary \ref{cor:Reciprocity} for $P_3 \times P_2$, where $k=3$, and where
\[L = 
\left(
\begin{array}{cc}
k^2-3k+3  &  k^3-6k^2+13k-10 \\
k^2-4k+5  &  k^3-6k^2+14k-13 \\
\end{array}
\right).
\]
The orientation of the dashed edges can be chosen arbitrarily. The number below the graph gives the multiplicity of the given case. Therefore, the sum of the gray numbers (up to a sign) equals the evaluation of the second row sum of $L$ at $-3$
\begin{align*}
&(-3)^2 - 4 \cdot (-3)+5 + (-4)^3- 6\cdot(-3)^2 +14 \cdot (-3) - 13 = - 110 =\\
 = &(-1)\cdot (8+2+4+4+2+2+8+8+1+4+8+2+2+\\
 +& 1+1+2+4+1+8+2+2+8+2+4+4+8+8).
\end{align*}

%%%%%%%%%%%%%%%%%%%%%%%%%%%%%%%%%%%%%%%%%%%%%%%%%%
\begin{figure}[H]

\begin{tikzpicture}[scale= .65]
\draw [thick, ->](0,0) --(0,1);
\draw [thick,->] (0,1)--(0,2);
\draw[ thick,dashed](1,0) --(1,1);
\draw[thick,->] (1,1)--(1,2);
\draw [ thick,->]  (0,0) --(1,0);
\draw [ thick,dashed]  (1,1) --(0,1);
\draw [ thick,dashed]  (0,2) -- (1,2);
\node [left] at (0,2) {$1$};
\node [left] at (0,1) {$2$};
\node [left] at (0,0) {$3$};
\node [right] at (1,2) {$1$};
\node [right] at (1,1) {$2$};
\node [right] at (1,0) {$2$};
\node[below, gray] at (0.5,0){$8$};

\draw [thick, ->](2,0) --(2,1);
\draw [thick,->] (2,1)--(2,2);
\draw[ thick,<-](3,0) --(3,1);
\draw[thick,dashed] (3,1)--(3,2);
\draw [ thick,->]  (2,0) --(3,0);
\draw [ thick,->]  (3,1) --(2,1);
\draw [ thick,<-]  (2,2) -- (3,2);
\node [left] at (2,2) {$1$};
\node [left] at (2,1) {$2$};
\node [left] at (2,0) {$3$};
\node [right] at (3,2) {$3$};
\node [right] at (3,1) {$3$};
\node [right] at (3,0) {$1$};
\node[below, gray] at (2.5,0){$2$};

\draw [thick, ->](4,0) --(4,1);
\draw [thick,->] (4,1)--(4,2);
\draw[ thick,->](5,0) --(5,1);
\draw[thick,<-](5,1)--(5,2);
\draw [ thick,dashed]  (4,0) --(5,0);
\draw [ thick]  (5,1) --(4,1);
\draw [ thick,<-]  (4,2) -- (5,2);
\node [left] at (4,2) {$1$};
\node [left] at (4,1) {$2$};
\node [left] at (4,0) {$3$};
\node [right] at (5,2) {$3$};
\node [right] at (5,1) {$1$};
\node [right] at (5,0) {$3$};
\node[below, gray] at (4.5,0){$2$};

\draw [thick, ->](6,0) --(6,1);
\draw [thick,->] (6,1)--(6,2);
\draw[ thick,dashed](7,0) --(7,1);
\draw[thick,->] (7,1)--(7,2);
\draw [ thick,dashed]  (6,0) --(7,0);
\draw [ thick,->]  (7,1) --(6,1);
\draw [ thick,dashed]  (6,2) -- (7,2);
\node [left] at (6,2) {$1$};
\node [left] at (6,1) {$2$};
\node [left] at (6,0) {$3$};
\node [right] at (7,2) {$1$};
\node [right] at (7,1) {$3$};
\node [right] at (7,0) {$3$};
\node[below, gray] at (6.5,0){$8$};

\draw [thick, ->](8,0) --(8,1);
\draw [thick,->] (8,1)--(8,2);
\draw[ thick,<-](9,0) --(9,1);
\draw[thick,dashed] (9,1)--(9,2);
\draw [ thick,->]  (8,0) --(9,0);
\draw [ thick,->]  (9,1) --(8,1);
\draw [ thick,<-]  (8,2) -- (9,2);
\node [left] at (8,2) {$1$};
\node [left] at (8,1) {$2$};
\node [left] at (8,0) {$3$};
\node [right] at (9,2) {$3$};
\node [right] at (9,1) {$3$};
\node [right] at (9,0) {$2$};
\node[below, gray] at (8.5,0){$2$};

\draw [thick, ->](10,0) --(10,1);
\draw [thick,->] (10,1)--(10,2);
\draw[ thick,->](11,0) --(11,1);
\draw[thick,<-] (11,1)--(11,2);
\draw [ thick,dashed]  (10,0) --(11,0);
\draw [ thick,dashed]  (11,1) --(10,1);
\draw [ thick,<-]  (10,2) -- (11,2);
\node [left] at (10,2) {$1$};
\node [left] at (10,1) {$2$};
\node [left] at (10,0) {$3$};
\node [right] at (11,2) {$3$};
\node [right] at (11,1) {$2$};
\node [right] at (11,0) {$3$};
\node[below, gray] at (10.5,0){$4$};

\draw [thick, ->](12,0) --(12,1);
\draw [thick,->] (12,1)--(12,2);
\draw[ thick,dashed] (13,0) --(13,1);
\draw[thick,->] (13,1)--(13,2);
\draw [ thick,dashed]  (12,0) --(13,0);
\draw [ thick,->]  (13,1) --(12,1);
\draw [ thick,<-]  (12,2) -- (13,2);
\node [left] at (12,2) {$1$};
\node [left] at (12,1) {$2$};
\node [left] at (12,0) {$3$};
\node [right] at (13,2) {$2$};
\node [right] at (13,1) {$3$};
\node [right] at (13,0) {$3$};
\node[below, gray] at (12.5,0){$4$};

\draw [thick, ->](14,0) --(14,1);
\draw [thick,->] (14,1)--(14,2);
\draw[ thick,dashed](15,0) --(15,1)--(15,2);
\draw [ thick,->]  (14,0) --(15,0);
\draw [ thick,dashed]  (15,1) --(14,1);
\draw [ thick,<-]  (14,2) -- (15,2);
\node [left] at (14,2) {$1$};
\node [left] at (14,1) {$2$};
\node [left] at (14,0) {$3$};
\node [right] at (15,2) {$2$};
\node [right] at (15,1) {$2$};
\node [right] at (15,0) {$2$};
\node[below, gray] at (14.5,0){$8$};

\draw [thick, ->](16,0) --(16,1);
\draw [thick,->] (16,1)--(16,2);
\draw[ thick,dashed](17,0) --(17,1)--(17,2);
\draw [ thick,dashed]  (16,0) --(17,0);
\draw [ thick,->]  (17,1) --(16,1);
\draw [ thick,<-]  (16,2) -- (17,2);
\node [left] at (16,2) {$1$};
\node [left] at (16,1) {$2$};
\node [left] at (16,0) {$3$};
\node [right] at (17,2) {$3$};
\node [right] at (17,1) {$3$};
\node [right] at (17,0) {$3$};
\node[below, gray] at (16.5,0){$8$};

%%%%%%%Next Line 
\draw [thick, ->](0,4) --(0,5);
\draw [thick,->] (0,5)--(0,6);
\draw[ thick,dashed](1,4) --(1,5);
\draw[thick,<-] (1,5)--(1,6);
\draw [ thick,->]  (0,4) --(1,4);
\draw [ thick,dashed]  (1,5) --(0,5);
\draw [ thick,<-]  (0,6) -- (1,6);
\node [left] at (0,6) {$1$};
\node [left] at (0,5) {$2$};
\node [left] at (0,4) {$3$};
\node [right] at (1,6) {$3$};
\node [right] at (1,5) {$2$};
\node [right] at (1,4) {$2$};
\node[below, gray] at (.5,4){$4$};

\draw [thick, ->](2,4) --(2,5);
\draw [thick,->] (2,5)--(2,6);
\draw[ thick,->](3,4) --(3,5);
\draw[thick,->](3,5)--(3,6);
\draw [ thick,dashed]  (2,4) --(3,4);
\draw [ thick,dashed]  (3,5) --(2,5);
\draw [ thick,dashed]  (2,6) -- (3,6);
\node [left] at (2,6) {$1$};
\node [left] at (2,5) {$2$};
\node [left] at (2,4) {$3$};
\node [right] at (3,6) {$1$};
\node [right] at (3,5) {$2$};
\node [right] at (3,4) {$3$};
\node[below, gray] at (2.5,4){$8$};

\draw [thick, ->](4,4) --(4,5);
\draw [thick,->] (4,5)--(4,6);
\draw[ thick,<-] (5,4) --(5,5);
\draw[thick,->] (5,5)--(5,6);
\draw [ thick,->]  (4,4) --(5,4);
\draw [ thick,->]  (5,5) --(4,5);
\draw [ thick,dashed]  (4,6) -- (5,6);
\node [left] at (4,6) {$1$};
\node [left] at (4,5) {$2$};
\node [left] at (4,4) {$3$};
\node [right] at (5,6) {$1$};
\node [right] at (5,5) {$3$};
\node [right] at (5,4) {$2$};
\node[below, gray] at (4.5,4){$2$};

\draw [thick, ->](6,4) --(6,5);
\draw [thick,->] (6,5)--(6,6);
\draw[ thick,->](7,4) --(7,5);
\draw[thick,<-] (7,5)--(7,6);
\draw [ thick,dashed]  (6,4) --(7,4);
\draw [ thick,<-]  (7,5) --(6,5);
\draw [ thick,<-]  (6,6) -- (7,6);
\node [left] at (6,6) {$1$};
\node [left] at (6,5) {$2$};
\node [left] at (6,4) {$3$};
\node [right] at (7,6) {$2$};
\node [right] at (7,5) {$1$};
\node [right] at (7,4) {$3$};
\node[below, gray] at (6.5,4){$2$};

\draw [thick, ->](8,4) --(8,5);
\draw [thick,->] (8,5)--(8,6);
\draw[ thick,<-] (9,4) --(9,5);
\draw[ thick,->] (9,5)--(9,6);
\draw [ thick,->]  (8,4) --(9,4);
\draw [ thick,->]  (9,5) --(8,5);
\draw [ thick,<-]  (8,6) -- (9,6);
\node [left] at (8,6) {$1$};
\node [left] at (8,5) {$2$};
\node [left] at (8,4) {$3$};
\node [right] at (9,6) {$2$};
\node [right] at (9,5) {$3$};
\node [right] at (9,4) {$1$};
\node[below, gray] at (8.5,4){$1$};

\draw [thick, ->](10,4) --(10,5);
\draw [thick,->] (10,5)--(10,6);
\draw[ thick,->](11,4) --(11,5);
\draw[ thick,<-] (11,5)--(11,6);
\draw [ thick,->]  (10,4) --(11,4);
\draw [ thick,<-]  (11,5) --(10,5);
\draw [ thick,<-]  (10,6) -- (11,6);
\node [left] at (10,6) {$1$};
\node [left] at (10,5) {$2$};
\node [left] at (10,4) {$3$};
\node [right] at (11,6) {$3$};
\node [right] at (11,5) {$1$};
\node [right] at (11,4) {$2$};
\node[below, gray] at (10.5,4){$1$};

\draw [thick, ->](12,4) --(12,5);
\draw [thick,->] (12,5)--(12,6);
\draw[ thick,<-](13,4) --(13,5);
\draw[ thick,<-] (13,5) --(13,6);
\draw [ thick,->]  (12,4) --(13,4);
\draw [ thick,dashed]  (13,5) --(12,5);
\draw [ thick,<-]  (12,6) -- (13,6);
\node [left] at (12,6) {$1$};
\node [left] at (12,5) {$2$};
\node [left] at (12,4) {$3$};
\node [right] at (13,6) {$3$};
\node [right] at (13,5) {$2$};
\node [right] at (13,4) {$1$};
\node[below, gray] at (12.5,4){$2$};

\draw [thick, ->](14,4) --(14,5);
\draw [thick,->] (14,5)--(14,6);
\draw[ thick,<-](15,4) --(15,5);
\draw[ thick,dashed](15,5) --(15,6);
\draw [ thick,->]  (14,4) --(15,4);
\draw [ thick,dashed]  (15,5) --(14,5);
\draw [ thick,<-]  (14,6) -- (15,6);
\node [left] at (14,6) {$1$};
\node [left] at (14,5) {$2$};
\node [left] at (14,4) {$3$};
\node [right] at (15,6) {$2$};
\node [right] at (15,5) {$2$};
\node [right] at (15,4) {$1$};
\node[below, gray] at (14.5,4){$4$};

\draw [thick, ->](16,4) --(16,5);
\draw [thick,->] (16,5)--(16,6);
\draw[ thick,->](17,4) --(17,5);
\draw[ thick,<-] (17,5)--(17,6);
\draw [ thick,->]  (16,4) --(17,4);
\draw [ thick,<-]  (17,5) --(16,5);
\draw [ thick,<-]  (16,6) -- (17,6);
\node [left] at (16,6) {$1$};
\node [left] at (16,5) {$2$};
\node [left] at (16,4) {$3$};
\node [right] at (17,6) {$2$};
\node [right] at (17,5) {$1$};
\node [right] at (17,4) {$2$};
\node[below, gray] at (16.5,4){$1$};

%%%%%%%Next Line 
\draw [thick, ->](0,8) --(0,9);
\draw [thick,->] (0,9)--(0,10);
\draw[ thick,dashed](1,8) --(1,9)--(1,10);
\draw [ thick,->]  (0,8) --(1,8);
\draw [ thick,<-]  (1,9) --(0,9);
\draw [ thick,dashed]  (0,10) -- (1,10);
\node [left] at (0,10) {$1$};
\node [left] at (0,9) {$2$};
\node [left] at (0,8) {$3$};
\node [right] at (1,10) {$1$};
\node [right] at (1,9) {$1$};
\node [right] at (1,8) {$1$};
\node[below, gray] at(.5,8) {$8$};

\draw [thick, ->](2,8) --(2,9);
\draw [thick,->] (2,9)--(2,10);
\draw[ thick,dashed](3,8) --(3,9);
\draw [thick,<-](3,9) --(3,10);
\draw [ thick,->]  (2,8) --(3,8);
\draw [ thick,->]  (3,9) --(2,9);
\draw [ thick,<-]  (2,10) -- (3,10);
\node [left] at (2,10) {$1$};
\node [left] at (2,9) {$2$};
\node [left] at (2,8) {$3$};
\node [right] at (3,10) {$2$};
\node [right] at (3,9) {$1$};
\node [right] at (3,8) {$1$};
\node[below, gray] at (2.5,8){$2$};

\draw [thick, ->](4,8) --(4,9);
\draw [thick,->] (4,9)--(4,10);
\draw[ thick,<-](5,8) --(5,9);
\draw[thick,->](5,9)--(5,10);
\draw [ thick,->]  (4,8) --(5,8);
\draw [ thick,dashed]  (5,9) --(4,9);
\draw [ thick,dashed]  (4,10) -- (5,10);
\node [left] at (4,10) {$1$};
\node [left] at (4,9) {$2$};
\node [left] at (4,8) {$3$};
\node [right] at (5,10) {$1$};
\node [right] at (5,9) {$2$};
\node [right] at (5,8) {$1$};
\node[below, gray] at (4.5,8){$4$};

\draw [thick, ->](6,8) --(6,9);
\draw [thick,->] (6,9)--(6,10);
\draw[ thick,->](7,8) --(7,9);
\draw[thick,dashed]  (7,9) --(7,10);
\draw [ thick,->]  (6,8) --(7,8);
\draw [ thick,<-]  (7,9) --(6,9);
\draw [ thick,dashed]  (6,10) -- (7,10);
\node [left] at (6,10) {$1$};
\node [left] at (6,9) {$2$};
\node [left] at (6,8) {$3$};
\node [right] at (7,10) {$1$};
\node [right] at (7,9) {$1$};
\node [right] at (7,8) {$2$};
\node[below, gray] at (6.5,8){$4$};

\draw [thick, ->](8,8) --(8,9);
\draw [thick,->] (8,9)--(8,10);
\draw[ thick,dashed](9,8) --(9,9);
\draw[ thick,<-] (9,9)--(9,10);
\draw [ thick,->]  (8,8) --(9,8);
\draw [ thick,->]  (9,9) --(8,9);
\draw [ thick,<-]  (8,10) -- (9,10);
\node [left] at (8,10) {$1$};
\node [left] at (8,9) {$2$};
\node [left] at (8,8) {$3$};
\node [right] at (9,10) {$3$};
\node [right] at (9,9) {$1$};
\node [right] at (9,8) {$1$};
\node[below, gray] at (8.5,8){$2$};

\draw [thick, ->](10,8) --(10,9);
\draw [thick,->] (10,9)--(10,10);
\draw[ thick,<-](11,8) --(11,9);
\draw[ thick,->] (11,9)--(11,10);
\draw [ thick,->]  (10,8) --(11,8);
\draw [ thick,->]  (11,9) --(10,9);
\draw [ thick,dashed]  (10,10) -- (11,10);
\node [left] at (10,10) {$1$};
\node [left] at (10,9) {$2$};
\node [left] at (10,8) {$3$};
\node [right] at (11,10) {$1$};
\node [right] at (11,9) {$3$};
\node [right] at (11,8) {$1$};
\node[below, gray] at (10.5,8){$2$};

\draw [thick, ->](12,8) --(12,9);
\draw [thick,->] (12,9)--(12,10);
\draw[ thick,->](13,8) --(13,9);
\draw[ thick,dashed]  (13,9)--(13,10);
\draw [ thick,dashed]  (12,8) --(13,8);
\draw [ thick,<-]  (13,9) --(12,9);
\draw [ thick,dashed]  (12,10) -- (13,10);
\node [left] at (12,10) {$1$};
\node [left] at (12,9) {$2$};
\node [left] at (12,8) {$3$};
\node [right] at (13,10) {$1$};
\node [right] at (13,9) {$1$};
\node [right] at (13,8) {$3$};
\node[below, gray] at (12.5,8){$8$};

\draw [thick, ->](14,8) --(14,9);
\draw [thick,->] (14,9)--(14,10);
\draw[ thick,->](15,8) --(15,9);
\draw[ thick,dashed] (15,9)--(15,10);
\draw [ thick,dashed]  (14,8) --(15,8);
\draw [ thick,dashed]  (15,9) --(14,9);
\draw [ thick,<-]  (14,10) -- (15,10);
\node [left] at (14,10) {$1$};
\node [left] at (14,9) {$2$};
\node [left] at (14,8) {$3$};
\node [right] at (15,10) {$2$};
\node [right] at (15,9) {$2$};
\node [right] at (15,8) {$3$};
\node[below, gray] at (14.5,8){$8$};

\draw [thick, ->](16,8) --(16,9);
\draw [thick,->] (16,9)--(16,10);
\draw[ thick,<-](17,8) --(17,9);
\draw[ thick,->] (17,9)--(17,10);
\draw [ thick,->]  (16,8) --(17,8);
\draw [ thick,->]  (17,9) --(16,9);
\draw [ thick,<-]  (16,10) -- (17,10);
\node [left] at (16,10) {$1$};
\node [left] at (16,9) {$2$};
\node [left] at (16,8) {$3$};
\node [right] at (17,10) {$2$};
\node [right] at (17,9) {$3$};
\node [right] at (17,8) {$2$};
\node[below, gray] at (16.5,8){$1$};

\end{tikzpicture}
\caption{All $(1,2,3)$-restricted, compatible pairs of acyclic orientations and (not necessary proper) colorings of $P_3 \times P_2$.}\label{fig:restricted reciprocity}
\end{figure}
\end{example}

%%%%%%%%%%%%%%%%%%%%%%%%%%%%%%%%%%%%%%%%%%%%%%%%%%

If however one is interested in the asymptotic behavior of graphs $G \times C_n$, then we need to find good bounds for the biggest eigenvalue of $L$, as this dominates the asymptotics. The next theorem gives an upper and a lower bound and in particular it shows that the eigenvalue grows like a polynomial of degree $\# V(G)$.
\begin{lemma}
\label{UpperLowerBounds}
Let $\lambda_{\max}$ be the biggest eigenvalue of $L$ (and thus of the adjacency matrix $A_{M_G}$). Then 
\begin{equation}
\label{eq:UpperLowerBounds}
\delta(L) \leq \lambda_{\max} \leq \Delta(L) \text{,}
\end{equation}
where $\delta(L)$ and $\Delta(L)$ are the smallest and biggest row sum of $L$, respectively.
\end{lemma}

\begin{proof}
The biggest eigenvalue of the adjacency matrix of a graph is bounded above and below by the biggest and smallest degree of the graph, respectively, see \cite[Thm. 1]{MR1077731}. These are exactly the biggest and smallest row sums of $L$.
\end{proof}

Now the question of determining the biggest eigenvalue reduces to determining the smallest and biggest row sum of $L$. This might be computationally challenging. Our next result gives a combinatorial interpretation of the two highest coefficients, which gives us a quick way to obtain the two highest coefficients without computing $L$. 
\begin{lemma}
\label{cor:FastBound}
Let $\tilde{o}_1, \tilde{o}_2,\dots, \tilde{o}_p$ be orbits as defined in Definition \ref{def:orbits}. Let $\tilde{o}_p$ be the orbit using $N = \# V(G)$ colors. Then $\delta (L)$ and $\Delta (L)$ are polynomials of degree $N$,  their leading coefficient is $a_N = 1$, and 
\[
a_{N-1} = - F + \# \text{orbits using } N-1 \text{ colors,}
\] 
where 
\[
F = \# \text{edges of } G \times P_2 \text{that are not edges of } G \times \{ 1\}.
\]
In particular, the highest two terms of both polynomials agree.
\end{lemma}

\begin{proof}
This directly follows from Theorem \ref{thm:restrictedReciprocity}.
\end{proof}
To summarize, we get:
\begin{proposition}
\label{prop:AsymptoticsCn}
Let $G$ be a graph and $N = \# V(G)$ and let $\delta (L)$ and $\Delta (L)$ be as above. Then the doubly asymptotic behavior of the number of proper $k$-colorings of $G \times C_n$ is dominated by $\lambda_{\max}^{n-1}$ and
\[
\delta(L) \leq \lambda_{\max} \leq \Delta(L),
\]
where $\delta(L)  = \sum_{i=0}^N a_i k^i $, $\Delta(L)=\sum_{i=0}^N b_i k^i $, $a_N = b_N$, and $a_{N-1} = b_{N-1}$.
\end{proposition}

\subsection{Results on Orbit Counting}
As we have seen in the previous section, we should exploit the symmetry of our problem. However, what we did not address so far is how many orbits we can expect. This subsection only deals with finding the explicit number of orbits after quotienting by permutations of colors. We identify two colorings $c_1$ and $c_2$ if there is a permutation $\pi$ of colors such that $\pi (c_1) = c_2$. We want to count the number of orbits of $G$ after taking the quotient by permutations of colors.

Let 
\[ 
F_k (G) := \# \text{ orbits of $k$-colorings of G} 
\]
be the function that counts the number of orbits. Note that $F_{k} (G) = F_{k+l}(G)$ if $k \geq |V|$ and $l \in \mathbb{Z}_{\geq 0}$. Therefore, we can define $F(G): = F_k (G)$ for a $k\geq |V|$. We didn't find the following lemma in the literature, but it is quite possibly known.

\begin{lemma}[Deletion-Contraction for orbits]
\label{lem:DelCon}
Let $e \in E$ be given and let $G/e$ denote the contraction of $G$ along the edge $e$. Then
\begin{equation}
\label{eq:DelCon}
F(G - e) = F(G) + F(G/e).
\end{equation}
\end{lemma}
\begin{proof}
Let $k\geq |V|$. This proof is analogous to the proof of deletion-contraction for the chromatic polynomial. This proof establishes a bijection $b$ between the sets
\begin{equation}
\label{eq:bijection}
\stackrel{A:=} {\overbrace{\left\{ k \text{-colorings of }G -e \right\}}} \stackrel{b}{\longleftrightarrow} \stackrel{B:=}{\overbrace{\left\{ k \text{-colorings of }G\right\}}} \sqcup \stackrel{C:=}{\overbrace{\left\{ k \text{-colorings of }G/e \right\}}} \text{,}
\end{equation}
where clearly $B \cap C = \emptyset$. To prove $(\ref{eq:DelCon})$, we need to show that if $c_1$, $c_2 \in A$ are in the same orbit, then the corresponding colorings $c'_1$, $c'_2 \in B \cup C$ are also in the same orbit, and vice versa.

Let $u$ and $v$ be the vertices connected by the edge $e$ and let $c_1$, $c_2 \in A$ be two colorings in the same orbit. Then there are two cases:
\begin{enumerate}
\item If $c_i(u) \neq c_i(v)$ where $i \in \{1,2\}$, then $b(c_i) \in B$, so the associated coloring is a proper coloring of $G$. But this also means that $\pi(b(c_1)) = b(c_2)$, since the bijection does not change the coloring.
\item If $c_i(u) = c_i(v)$, then $b(c_i) \in C$, so the associated coloring is a proper coloring of $G/e$. This again directly implies that $\pi(b(c_1)) = b(c_2)$, since $b$ does not change the coloring of $u$ and $v$.
\end{enumerate}

The other direction follows analogously.

\end{proof}
\begin{remark}
Since we do not use graph symmetry, the following results give a crude upper bound for the dimension of the matrix $L_G$.
\end{remark}
Since we can now use deletion-contraction to count the number of orbits, we can successively remove all edges. This means that the number of orbits of a graph without edges is of special interest. As it turns out, the number of orbits of these graphs are counted by the Bell numbers $\mathcal{B}_n$, where \[\mathcal{B}_n := \# \text{number of partitions of a set of size }n.\]
\begin{lemma}
Let $G_n = ([n], \emptyset)$ be the graph on $n$ nodes without edges. Then
\begin{equation}
F(G_n) = \mathcal{B}_n \text{, }
\end{equation}
 
where $\mathcal{B}_n$ is the $n^{\text{th}}$ Bell number.
\end{lemma}
\begin{proof}
Since we do not take graph symmetries into account, a coloring of a graph without edges corresponds to a partition of a set with $n$ \emph{labeled} elements, and every such set partition corresponds to a coloring. These set partitions are counted by the Bell numbers.
\end{proof}
Bell numbers do not only appear in counting problems, but $\mathcal{B}_n$ also appears as the $n^{\text{th}}$ moment in the \emph{Poisson distribution} with mean $1$. We want to finish this section by giving some small applications of deletion-contraction to count the number of orbits of path graphs and cycle graphs.
\begin{lemma}
\label{Pathorbits}
Let $P_n$ be the path graph on $n$ nodes. The number of orbits $F(P_n)$ is 
\begin{equation}
\label{eq:pathorbits}
F(P_n) = \mathcal{B}_{n-1} \text{.}
\end{equation}
\end{lemma}
\begin{proof}

This follows from the special case $k=1$ from Theorem 2, \cite{Yang96}.
\end{proof}

As a corollary, we obtain the following relation that we also didn't find in the literature.
\begin{corollary}
The Bell numbers satisfy the relation
\begin{equation}
\label{eq:rec}
\mathcal{B}_n = \sum_{i=1}^{n+1} \binom{n}{i-1}(-1)^{n+1-i}\mathcal{B}_i\text{.}
\end{equation}
\end{corollary}

\begin{proof}
Use (\ref{eq:DelCon}) repeatedly on $P_n$ and induct on $n$.
\end{proof}

\begin{corollary}
Let $C_n$ denote the cycle graph on $n$ nodes. Then
\begin{equation}
\label{eq:cycleorbit}
F(C_n) = \mathcal{B}_{n-1} - \mathcal{B}_{n-2} + \dots + (-1)^n \mathcal{B}_1
\end{equation}
\end{corollary}
\begin{proof}
Apply (\ref{eq:DelCon}) to $C_n$ and induct on $n$.
\end{proof}
\begin{remark}
Deletion-contraction does \emph{not} work if we mod out by graph symmetries, as the bijection $b$ in (\ref{eq:bijection}) does no longer hold. On the right-hand side of (\ref{eq:bijection}), we have removed an edge or identified two nodes, which changes the automorphism group. However, one should in principle be able to determine the number of orbits using Burnside's lemma.
\end{remark}

\section{Counting Discrete Markov Processes}
\label{sec:DMP}
\subsection{The General Setup}

In this section, we want to show how we can use a similar approach to obtain results about counting discrete Markov chains. We will first formulate the result in a general form before we apply it to order polytopes of what we call \emph{stacked posets}.

For us, a \emph{discrete Markov chain} is a discrete-time stochastic process with random variables $\{X_1, X_2, \dots, X_m \}$, where the random variables take values in a discrete set of states $\mathcal{S} $. Moreover, we require that a transition from state $s^i$ to state $s^j$ only depends on state $s^i$ and not on any previous states. A \emph{base state} is a labeling $b \in \mathbb{Z}^m$ of $m$ nodes by elements from $[m]$ with $\min_{i \in [m] } \{b_i \}=1$. A \emph{state} is an element in the equivalence class $b + k\mathbf{1}$, where $\mathbf{1}: = (1,\dots, 1)^t$ and $k \in \mathbb{Z}_{>0}$. A \emph{move} $f$ is a vector $f\in \mathbb{Z}^m$ and we say that $f$ is a move from state $s^i$ to $s^j$ if $r^i+f \in s^j$, where $r^i \in s^i$ and we denote it by  $s^i \stackrel{f}{\rightarrow} s^j$ . We assume that $|f_i|\leq m$ for all $i\in [m]$. Now let $\mathcal{S} = \left\{s^{1}, s^{2}, \dots, s^{r}  \right\}$ be a finite set of states with corresponding base states $\left\{b^1, b^2, \dots, b^r \right\}$ and let $\mathcal{F}$ be a finite set of moves. We define $M \in R^{r \times r}$ to be
\begin{equation}
\label{eq:MarkovMatrixDef}
M_{i,j} := M_{i,j}(x_1,x_2,\dots,x_r)\sum_{f\in \mathcal{F} \colon s^i \stackrel{f}{\rightarrow} s^j }x^f:=\sum_{f\in \mathcal{F} \colon s^i \stackrel{f}{\rightarrow} s^j } x_1^{f_1}\dots  x_r^{f_r}\text{,}
\end{equation} 
where $R:=\mathbb{Z}[\![ x_1^{\pm 1},x_2^{\pm 1}, \dots, x_r^{\pm 1}]\!]$ and we call $M$ a \emph{transition matrix}.

Similar to Theorem \ref{thm:L^n_entry}, $M_{i,j}^n$ records all chains of length $n+1$ with starting state $s^i$ and ending state $s^j$. Therefore, $M^n$ records all information on $n+1$ chains, and we thus get the following:

\begin{theorem}
\label{thm:MarkovMatrix}
With the previous notation, we have:

\begin{itemize}
\item the number of chains of length $n+1$ is given by
\begin{equation}
\label{eq:MarkovChain}
\# \text{number of chains of length n+1} = \mathbf{1}^t M^n \left|_{\mathbf x = (1,1,\dots,1)} \right. \mathbf{1}\text{,}
\end{equation}
\item the number $I_k$ of chains so that no element is increased by more than $k$ is given by
\begin{equation}
\label{eq:MarkovLocalChange}
I_k = \left(\deg_{x_1, x_2, \dots, x_r\leq k}(x^{b^1},\dots,x^{b^n}) M^n \right) \left|_{\mathbf x = (1,1,\dots,1)} \right. \mathbf{1}\text{,}
\end{equation}
where no indeterminate $x_i$ can have a degree bigger than $k$.
\end{itemize} 
\end{theorem}

\subsection{Application to Order Polytopes} 
In this section, we want to apply Theorem \ref{thm:MarkovMatrix} to a certain class of order polytopes which are geometric objects associated to posets which were famously introduced by Stanley \cite{StanleyOrderPoly}. A \emph{partially ordered set} ---or \emph{poset} for short--- $\mathcal{P}$ is a set together with a binary relation $\leq_{\mathcal{P}}$ that is reflexive, antisymmetric, and transitive, and $\leq_{\mathcal{P}}$ is called a \emph{partial order on $\mathcal{P}$}. If the poset is clear from the context, we will simply write $\leq$. All our posets will be \emph{finite}. To every poset $\mathcal{P}$, one can define the order polytope $\mathcal{O}(\mathcal{P})$:
\begin{definition}[\cite{StanleyOrderPoly}, Def. 1.1]
The order polytope $\mathcal{O}(\mathcal{P})$ of the poset $\mathcal{P}$ is the subset of $\mathbb{R}^{\mathcal{P}} = \{f\colon \mathcal{P} \rightarrow \mathbb{R} \}$ defined by the conditions
\begin{align*}
&0 \leq f(x) \leq 1  &\qquad \text{for all } x\in \mathcal{P}, \\
&f(x)\leq f(y)      &\qquad \text{if } x\leq y \text{ in }\mathcal{P}.\\
\end{align*}
\end{definition}
The geometry of order polytopes encodes combinatorial properties of the underlying poset. For instance, order filters of $\mathcal{P}$ are in bijection with the vertices of $\mathcal{O}(\mathcal{P})$. An \emph{order filter $F$} is a set such that $x\in F$ together with $y \leq x$ implies that $y \in F$. The vertices are then the indicator vectors of order filters. This implies that order polytopes are always lattice polytopes. Furthermore, order polytopes are also closely related to order-preserving maps.
\begin{definition}
A map $\pi \colon \mathcal{P} \longrightarrow \mathcal{Q}$ between two posets $\mathcal{P}$ and $\mathcal{Q}$ is called \emph{order-preserving} if $\pi(x) \leq_{\mathcal{Q}} \pi(y)$ for all $x \leq_{\mathcal{P}} y$.
\end{definition}
The following theorem connects the Ehrhart polynomial of order polytopes to order-preserving maps:
\begin{theorem}\cite[Theorem 4.1]{StanleyOrderPoly}
\label{thm:EhrPolyOrderPolytope}
Let $\mathcal{P}$ be a finite poset. The Ehrhart polynomial $E_{\mathcal{O}(\mathcal{P})}$ of $\mathcal{O}(\mathcal{P})$ is given by
\begin{equation}
\label{eq:EhrPolyOrderPolytope}
E_{\mathcal{O}(\mathcal{P})} (t) = \# \text{ order-preserving maps }\eta \colon \mathcal{P} \rightarrow [t+1].
\end{equation} 
\end{theorem}

We will apply the machinery from the previous subsection to a special class of posets, which we call \emph{stacked posets}. We will then define what the states are and what the set of possible moves is. 
\begin{definition}
Let $\mathcal{P}$ be a graded poset of rank $2$ satisfying 
\[
\# \{x \in \mathcal{P} \colon \, \deg x = 1 \} = \# \{x \in \mathcal{P} \colon \, \deg x = 2 \}. 
\]
We can label the elements of $\mathcal{P}$ by $\{ (1,1),\dots, (1,l),(2,1),\dots, (2,l)\}$, and we assume that $(1,k) \leq_{\mathcal{P}} (2,k)$ for all $k$. We define \emph{the stacked poset $\mathcal{P}_n$ of height $n+1$} to be the poset with elements $ \{(i,j) \colon i\in [n+1], j\in [l] \}$ and with cover relations
\begin{equation}
\label{eq:CoveringCondition}
(i+1, p) \gtrdot_{\mathcal{P}_n} (i,q) \text{ if and only if }(2,p) \gtrdot_{\mathcal{P}} (1,q).
\end{equation}
This directly implies
\begin{equation}
\label{eq:CardinalityCondition}
\# \left\{x \in \mathcal{P}_n \colon \deg x = 1 \right\} = \dots =\# \left\{x \in \mathcal{P}_n \colon \deg x = n+1 \right\} \text{.}
\end{equation}
The subposet $\mathcal{P} : = \left\{x \in \mathcal{P}_n \colon \deg x \in \{1,2 \} \right\}$ is called the \emph{base poset}.
\end{definition}

Since $\mathcal{P}_n$ is completely determined by $\mathcal{P}$, we will normally only give the definition of $\mathcal{P}$.
\begin{example}
\label{ex:StackedPoset}
The poset $\mathcal{P} = \{a,b,c,d,e,f, g, h\}$ with the relations $a,b \leq_{\mathcal{P}} c , d$; $c,d \leq_{\mathcal{P}} e,f$; and   $e,f \leq_{\mathcal{P}} g,h$ is a $3$-stacked poset and with base poset $B = \{a,b,c,d \}$. 
\end{example}
\begin{figure}[H]
\label{fig:stacked_posets}
\center
\begin{tikzpicture}[scale=1]
\node[circle,draw=black,fill=white!80!black,minimum size=20,scale=.5] (1) at (0,0) {(1,1)};
\node[circle,draw=black,fill=white!80!black,minimum size=20,scale=.5]  (2) at (1,0) {(1,2)};
\node[circle,draw=black,fill=white!80!black,minimum size=20,scale=.5]  (3) at (0,1) {(2,1)};
\node[circle,draw=black,fill=white!80!black,minimum size=20,scale=.5]  (4) at (1,1) {(2,2)};
\node[circle,draw=black,fill=white!80!black,minimum size=20,scale=.5]  (5) at (0,2) {(3,1)};
\node[circle,draw=black,fill=white!80!black,minimum size=20,scale=.5]  (6) at (1,2) {(3,2)};
\node[circle,draw=black,fill=white!80!black,minimum size=20,scale=.5]  (7) at (0,3) {(4,1)};
\node[circle,draw=black,fill=white!80!black,minimum size=20,scale=.5]  (8) at (1,3) {(4,2)};
\draw (1) -- (3);
\draw (1) -- (4);
\draw (2) -- (4);
\draw (3) -- (5);
\draw (3) -- (6);
\draw (4) -- (6);
\draw (5) -- (7);
\draw (5) -- (8);
\draw (6) -- (8);

\node[circle,draw=black,fill=white!80!black,minimum size=20,scale=.5] (1) at (2,0) {(1,1)};
\node[circle,draw=black,fill=white!80!black,minimum size=20,scale=.5]  (2) at (3,0) {(1,2)};
\node[circle,draw=black,fill=white!80!black,minimum size=20,scale=.5]  (3) at (2,1) {(2,1)};
\node[circle,draw=black,fill=white!80!black,minimum size=20,scale=.5]  (4) at (3,1) {(2,2)};
\node[circle,draw=black,fill=white!80!black,minimum size=20,scale=.5]  (5) at (2,2) {(3,1)};
\node[circle,draw=black,fill=white!80!black,minimum size=20,scale=.5]  (6) at (3,2) {(3,2)};
\node[circle,draw=black,fill=white!80!black,minimum size=20,scale=.5]  (7) at (2,3) {(4,1)};
\node[circle,draw=black,fill=white!80!black,minimum size=20,scale=.5]  (8) at (3,3) {(4,2)};
\draw (1) -- (3);
\draw (1) -- (4);
\draw (2) -- (3);
\draw (2) -- (4);
\draw (3) -- (5);
\draw (3) -- (6);
\draw (4) -- (5);
\draw (4) -- (6);
\draw (5) -- (7);
\draw (5) -- (8);
\draw (6) -- (7);
\draw (6) -- (8);
\end{tikzpicture}
\caption{Two stacked posets, $\mathcal{P}$ from Example \ref{ex:StackedPoset} is on the right.}
\end{figure}
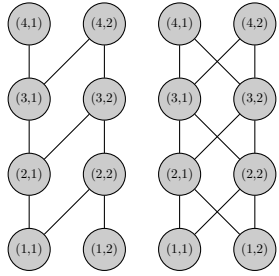
Next, we define the (base) states and moves of a stacked poset. 
\begin{definition}
A \emph{base state $b$} of a stacked poset is given by a labeling of all elements of same degree by numbers from $1$, $2$, $\dots$, $m$ such that if $j$ appears in the labeling, then all $1$, $2$, $\dots$, $j-1$ appear too, where $m = \# \left\{ x \colon \deg x = 1\right\}$. We identify this base state with the vector $b\in \left\{1,2,\dots,m \right\}^m$. A \emph{move from state $s^i$ to $s^j$} is adding a vector $f\in  \left\{0,1,2,\dots,m \right\}^m$ to $b^i$ such that
\begin{enumerate}
\item $b^i$ and $b^i+f$ contain all numbers from $1$ to $\max_{1 \leq k \leq m} (b^i +f)_k $,
\item $b^i +f \in s^j$,
\item and if we label the degree $a$ elements of $\mathcal{P}$ by $b^i$ and the degree $a+1$ elements by $b^i + f$, we get an order-preserving labeling.
\end{enumerate} 
The \emph{set of moves} $\mathcal{F}$ contains all possible moves.
\end{definition}

As we will see, this definition enables us to determine the number of surjective, order-preserving maps of $\mathcal{P}_n$ into $[k]$.
\begin{example}
\label{ex:TransitionMatrix}
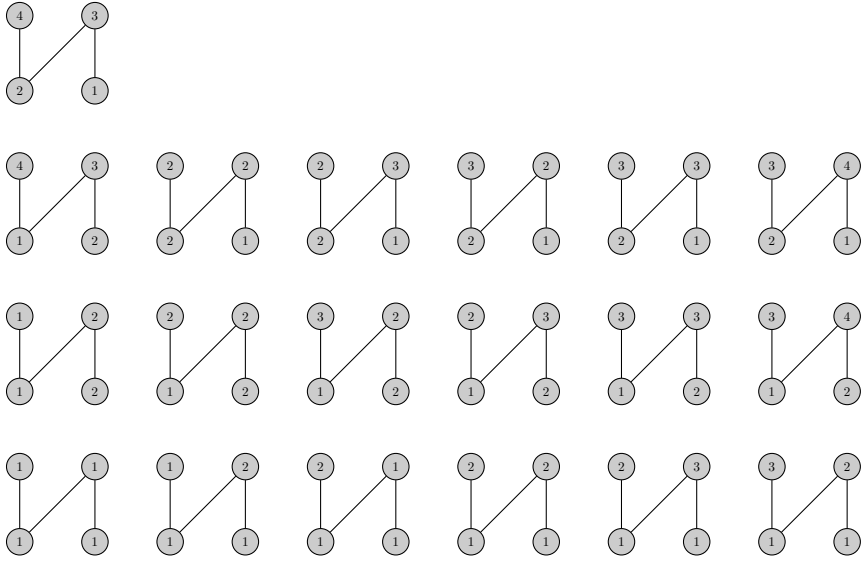
\begin{figure}[H]
\begin{tikzpicture}

\node[circle,draw=black,fill=white!80!black,minimum size=20,scale=.5] (1) at (0,0) {1};
\node[circle,draw=black,fill=white!80!black,minimum size=20,scale=.5]  (2) at (1,0) {1};
\node[circle,draw=black,fill=white!80!black,minimum size=20,scale=.5]  (3) at (0,1) {1};
\node[circle,draw=black,fill=white!80!black,minimum size=20,scale=.5]  (4) at (1,1) {1};
\draw (1) -- (3);
\draw (1) -- (4);
\draw (2) -- (4);

\node[circle,draw=black,fill=white!80!black,minimum size=20,scale=.5] (1) at (2,0) {1};
\node[circle,draw=black,fill=white!80!black,minimum size=20,scale=.5]  (2) at (3,0) {1};
\node[circle,draw=black,fill=white!80!black,minimum size=20,scale=.5]  (3) at (2,1) {1};
\node[circle,draw=black,fill=white!80!black,minimum size=20,scale=.5]  (4) at (3,1) {2};
\draw (1) -- (3);
\draw (1) -- (4);
\draw (2) -- (4);

\node[circle,draw=black,fill=white!80!black,minimum size=20,scale=.5] (1) at (4,0) {1};
\node[circle,draw=black,fill=white!80!black,minimum size=20,scale=.5]  (2) at (5,0) {1};
\node[circle,draw=black,fill=white!80!black,minimum size=20,scale=.5]  (3) at (4,1) {2};
\node[circle,draw=black,fill=white!80!black,minimum size=20,scale=.5]  (4) at (5,1) {1};
\draw (1) -- (3);
\draw (1) -- (4);
\draw (2) -- (4);

\node[circle,draw=black,fill=white!80!black,minimum size=20,scale=.5] (1) at (6,0) {1};
\node[circle,draw=black,fill=white!80!black,minimum size=20,scale=.5]  (2) at (7,0) {1};
\node[circle,draw=black,fill=white!80!black,minimum size=20,scale=.5]  (3) at (6,1) {2};
\node[circle,draw=black,fill=white!80!black,minimum size=20,scale=.5]  (4) at (7,1) {2};
\draw (1) -- (3);
\draw (1) -- (4);
\draw (2) -- (4);

\node[circle,draw=black,fill=white!80!black,minimum size=20,scale=.5] (1) at (8,0) {1};
\node[circle,draw=black,fill=white!80!black,minimum size=20,scale=.5]  (2) at (9,0) {1};
\node[circle,draw=black,fill=white!80!black,minimum size=20,scale=.5]  (3) at (8,1) {2};
\node[circle,draw=black,fill=white!80!black,minimum size=20,scale=.5]  (4) at (9,1) {3};
\draw (1) -- (3);
\draw (1) -- (4);
\draw (2) -- (4);

\node[circle,draw=black,fill=white!80!black,minimum size=20,scale=.5] (1) at (10,0) {1};
\node[circle,draw=black,fill=white!80!black,minimum size=20,scale=.5]  (2) at (11,0) {1};
\node[circle,draw=black,fill=white!80!black,minimum size=20,scale=.5]  (3) at (10,1) {3};
\node[circle,draw=black,fill=white!80!black,minimum size=20,scale=.5]  (4) at (11,1) {2};
\draw (1) -- (3);
\draw (1) -- (4);
\draw (2) -- (4);
%%%%%%%%%%%%%%%%%%%%%%%%%%%%%%%%%%END BOTTOM LINE%%%%%%%%%%%%%%%%%%%%%%%%%%%5

\node[circle,draw=black,fill=white!80!black,minimum size=20,scale=.5] (1) at (0,2) {1};
\node[circle,draw=black,fill=white!80!black,minimum size=20,scale=.5]  (2) at (1,2) {2};
\node[circle,draw=black,fill=white!80!black,minimum size=20,scale=.5]  (3) at (0,3) {1};
\node[circle,draw=black,fill=white!80!black,minimum size=20,scale=.5]  (4) at (1,3) {2};
\draw (1) -- (3);
\draw (1) -- (4);
\draw (2) -- (4);

\node[circle,draw=black,fill=white!80!black,minimum size=20,scale=.5] (1) at (2,2) {1};
\node[circle,draw=black,fill=white!80!black,minimum size=20,scale=.5]  (2) at (3,2) {2};
\node[circle,draw=black,fill=white!80!black,minimum size=20,scale=.5]  (3) at (2,3) {2};
\node[circle,draw=black,fill=white!80!black,minimum size=20,scale=.5]  (4) at (3,3) {2};
\draw (1) -- (3);
\draw (1) -- (4);
\draw (2) -- (4);

\node[circle,draw=black,fill=white!80!black,minimum size=20,scale=.5] (1) at (4,2) {1};
\node[circle,draw=black,fill=white!80!black,minimum size=20,scale=.5]  (2) at (5,2) {2};
\node[circle,draw=black,fill=white!80!black,minimum size=20,scale=.5]  (3) at (4,3) {3};
\node[circle,draw=black,fill=white!80!black,minimum size=20,scale=.5]  (4) at (5,3) {2};
\draw (1) -- (3);
\draw (1) -- (4);
\draw (2) -- (4);

\node[circle,draw=black,fill=white!80!black,minimum size=20,scale=.5] (1) at (6,2) {1};
\node[circle,draw=black,fill=white!80!black,minimum size=20,scale=.5]  (2) at (7,2) {2};
\node[circle,draw=black,fill=white!80!black,minimum size=20,scale=.5]  (3) at (6,3) {2};
\node[circle,draw=black,fill=white!80!black,minimum size=20,scale=.5]  (4) at (7,3) {3};
\draw (1) -- (3);
\draw (1) -- (4);
\draw (2) -- (4);

\node[circle,draw=black,fill=white!80!black,minimum size=20,scale=.5] (1) at (8,2) {1};
\node[circle,draw=black,fill=white!80!black,minimum size=20,scale=.5]  (2) at (9,2) {2};
\node[circle,draw=black,fill=white!80!black,minimum size=20,scale=.5]  (3) at (8,3) {3};
\node[circle,draw=black,fill=white!80!black,minimum size=20,scale=.5]  (4) at (9,3) {3};
\draw (1) -- (3);
\draw (1) -- (4);
\draw (2) -- (4);

\node[circle,draw=black,fill=white!80!black,minimum size=20,scale=.5] (1) at (10,2) {1};
\node[circle,draw=black,fill=white!80!black,minimum size=20,scale=.5]  (2) at (11,2) {2};
\node[circle,draw=black,fill=white!80!black,minimum size=20,scale=.5]  (3) at (10,3) {3};
\node[circle,draw=black,fill=white!80!black,minimum size=20,scale=.5]  (4) at (11,3) {4};
\draw (1) -- (3);
\draw (1) -- (4);
\draw (2) -- (4);

%%%%%%%%%%%%%%%%%%%%%%END SECOND LINE %%%%%%%%%%%%%%%%%%%%%%%%%%%%%%%

\node[circle,draw=black,fill=white!80!black,minimum size=20,scale=.5] (1) at (0,4) {1};
\node[circle,draw=black,fill=white!80!black,minimum size=20,scale=.5]  (2) at (1,4) {2};
\node[circle,draw=black,fill=white!80!black,minimum size=20,scale=.5]  (3) at (0,5) {4};
\node[circle,draw=black,fill=white!80!black,minimum size=20,scale=.5]  (4) at (1,5) {3};
\draw (1) -- (3);
\draw (1) -- (4);
\draw (2) -- (4);

\node[circle,draw=black,fill=white!80!black,minimum size=20,scale=.5] (1) at (2,4) {2};
\node[circle,draw=black,fill=white!80!black,minimum size=20,scale=.5]  (2) at (3,4) {1};
\node[circle,draw=black,fill=white!80!black,minimum size=20,scale=.5]  (3) at (2,5) {2};
\node[circle,draw=black,fill=white!80!black,minimum size=20,scale=.5]  (4) at (3,5) {2};
\draw (1) -- (3);
\draw (1) -- (4);
\draw (2) -- (4);

\node[circle,draw=black,fill=white!80!black,minimum size=20,scale=.5] (1) at (4,4) {2};
\node[circle,draw=black,fill=white!80!black,minimum size=20,scale=.5]  (2) at (5,4) {1};
\node[circle,draw=black,fill=white!80!black,minimum size=20,scale=.5]  (3) at (4,5) {2};
\node[circle,draw=black,fill=white!80!black,minimum size=20,scale=.5]  (4) at (5,5){3};
\draw (1) -- (3);
\draw (1) -- (4);
\draw (2) -- (4);

\node[circle,draw=black,fill=white!80!black,minimum size=20,scale=.5] (1) at (6,4) {2};
\node[circle,draw=black,fill=white!80!black,minimum size=20,scale=.5]  (2) at (7,4) {1};
\node[circle,draw=black,fill=white!80!black,minimum size=20,scale=.5]  (3) at (6,5) {3};
\node[circle,draw=black,fill=white!80!black,minimum size=20,scale=.5]  (4) at (7,5) {2};
\draw (1) -- (3);
\draw (1) -- (4);
\draw (2) -- (4);

\node[circle,draw=black,fill=white!80!black,minimum size=20,scale=.5] (1) at (8,4) {2};
\node[circle,draw=black,fill=white!80!black,minimum size=20,scale=.5]  (2) at (9,4) {1};
\node[circle,draw=black,fill=white!80!black,minimum size=20,scale=.5]  (3) at (8,5) {3};
\node[circle,draw=black,fill=white!80!black,minimum size=20,scale=.5]  (4) at (9,5) {3};
\draw (1) -- (3);
\draw (1) -- (4);
\draw (2) -- (4);

\node[circle,draw=black,fill=white!80!black,minimum size=20,scale=.5] (1) at (10,4) {2};
\node[circle,draw=black,fill=white!80!black,minimum size=20,scale=.5]  (2) at (11,4) {1};
\node[circle,draw=black,fill=white!80!black,minimum size=20,scale=.5]  (3) at (10,5) {3};
\node[circle,draw=black,fill=white!80!black,minimum size=20,scale=.5]  (4) at (11,5) {4};
\draw (1) -- (3);
\draw (1) -- (4);
\draw (2) -- (4);

\node[circle,draw=black,fill=white!80!black,minimum size=20,scale=.5] (1) at (0,6) {2};
\node[circle,draw=black,fill=white!80!black,minimum size=20,scale=.5]  (2) at (1,6) {1};
\node[circle,draw=black,fill=white!80!black,minimum size=20,scale=.5]  (3) at (0,7) {4};
\node[circle,draw=black,fill=white!80!black,minimum size=20,scale=.5]  (4) at (1,7) {3};
\draw (1) -- (3);
\draw (1) -- (4);
\draw (2) -- (4);

\end{tikzpicture}
\caption{A list of all transitions.}
\end{figure}

Let $\mathcal{P} = \{a,b,c,d,e,f\}$ with the relations $a \leq_{\mathcal{P}} c , d$; $b\leq_{\mathcal{P}} d$. The base states of $\mathcal{P}_n$ are $(1,1)$, $(1,2)$, and $(2,1)$. The set of moves contains the elements $(0,0)$, $(0,1)$, $(1,0)$ $(1,1)$, $(1,2)$, $(2,1)$, $(2,2)$, $(1,3)$, and $(3,1)$. Therefore, we get the transition matrix
\begin{equation*}
M = \begin{pmatrix}
 1 + xy   & x + xy^2       & x + x^2 y \\
 x + x^2y & 1+ xy + x^2y^2 & x^2 + x^3y \\
 y + xy^2 & y^2+ x y^3     & xy+x^2y^2 \\
\end{pmatrix}\text{.}
\end{equation*}
\end{example}

\begin{proposition}
\label{SurjOrderPreserving}
Let $\mathcal{P}_n$ be a stacked poset. The number of surjective, order-preserving maps from $\mathcal{P}_n$ into $[k]$ is given by
\begin{equation}
\label{eq:SurjOrderPreserving}
\# \left\{\pi\colon \mathcal{P}_n \twoheadrightarrow [k]  \colon \text{order-pres.} \right\}= \left(\operatorname{tdeg}_{= k} (x^{b^1},\dots, x^{b^r})M^{n}\right)\left|_{\mathbf x = (1,1,\dots,1)}, \right. \mathbf{1}
\end{equation}
where $\mathbf x = (x_1, \dots, x_m)$ and where $\operatorname{tdeg}_{= k}$ denotes the terms whose total degree equals $k$.
\end{proposition}

\begin{proof}
The set of moves $\mathcal{F}$ was defined so that we get an order-preserving map (condition (3)) and that the map is surjective (condition (1)). Therefore, \[(x^{b^1},\dots, x^{b^r})M^{n} \mathbf{1}\] contains all possible surjective, order-preserving maps. The terms of total degree $k$ correspond to all surjective, order-preserving maps into $[k]$.
\end{proof}
\begin{corollary}
\label{cor:OrderEhrPoly}
Let $\mathcal{P}_n$ be an $n$-stacked poset with base poset $\mathcal{P}$ and transition matrix $M$, then
\begin{equation}
\label{eq:OrderEhrPoly}
\operatorname{Ehr}_{\mathcal{O}(\mathcal{P}_n)}(k) = \sum_{i=1}^{k+1}\left. \begin{pmatrix}
k+1 \\ i
\end{pmatrix}  \left(\operatorname{tdeg}_{= i} (x^{b^1},\dots, x^{b^r})M^{n} \right)\right|_{\mathbf x = (1,1,\dots,1)}  \mathbf{1},
\end{equation}
where $\mathbf x = (x_1, \dots, x_m)$ .
\end{corollary}
\begin{remark}
\label{rem:EhrPositive}
This together with Proposition~\ref{prop:StructureM} shows that Ehrhart polynomials of order polytopes of stacked posets are Ehrhart positive, i. e., the coefficients of the Ehrhart polynomial in the monomial basis are positive.
\end{remark}
\begin{proof}
By Theorem \ref{thm:EhrPolyOrderPolytope}, the Ehrhart polynomial $\operatorname{Ehr}_{\mathcal{O}(\mathcal{P}_n)}(k)$ counts the number of order-preserving maps $\pi \colon \mathcal{P}_n \longrightarrow [k+1]$. However, we can see that
\begin{equation*}
\left \{ \pi  \colon \mathcal{P}_n \rightarrow [k+1] \colon \text{order-pres.}\right\} = \bigsqcup_{i=1}^{k+1}\left\{ \pi  \colon \mathcal{P}_n \twoheadrightarrow I \colon \text{ord-pres. } I\subset [k+1], \# I = i\right\} 
\end{equation*}
The claim now follows using Proposition \ref{SurjOrderPreserving}.
\end{proof}
\begin{example}
Let $\mathcal{P}_n$ be a chain of length $n$. Then $\mathcal{P}_n$ is a stacked poset, and $\mathcal{O}(\mathcal{P}_n)$ is unimodularly equivalent to the standard $n$-simplex. Thus, we have

\[\operatorname{Ehr}_{\mathcal{O}(\mathcal{P}_n)} = \binom{n+k}{k}.\]
On the other hand, applying Corollary \ref{cor:OrderEhrPoly}, we obtain the identity
\[
\operatorname{Ehr}_{\mathcal{O}(\mathcal{P}_n)} = \binom{n+k}{k} = \sum_{i=1}^{k+1} \binom{k+1}{i} \binom{n-1}{i-1}.
\]
\end{example}

We want to finish this section by stating explicit formulas for the entry $M_{i,j}$ where $i$ and $j$ correspond to the states $s^i$ and $s^j$, respectively. Let $S_{<_a}$ be the set of elements that are covered by an element $a$ in the given poset. 
\begin{proposition}
\label{prop:StructureM}
The entry $(i,j)$ of the transition matrix $M$ is given by 
\begin{equation*}
M_{i,j} = \sum_{\mu =0}^{\mu_{\max}} x_1^{a_1-b_1} \cdot \dots \cdot x_m^{a_m-b_m} \cdot (x_1\cdot \dots \cdot x_m)^{\lambda+ \mu}\text{,}
\end{equation*} 
where
\begin{itemize}
\item $a =(a_1, \dots, a_m)$ is the base state of $s^j$ and $b=(b_1, \dots, b_m)$ is the base state of $s^i$,
\item $\lambda = - \min \left(\left\{a_1 - b_p \colon p \in S_{<a_1} \right\}\cup \dots \cup\left\{a_m - b_p\colon p\in S_{<a_m}\right\}\right)$,
\item and $\mu_{\max} = m - \max \{a_1,\dots, a_m\} -\lambda + \max \{b_1, \dots, b_m\}$.
\end{itemize}
\end{proposition}

\begin{proof}
Let us first start by determining the entry $(1,i)$ where $1$ corresponds to the state $(1,\dots,1)$. We assume that $(a_1,\dots, a_m)$ is chosen so that $\min \{a_i \}=1$. The number of new colors introduced by $a$ is exactly $$\max \{a_1,\dots, a_m \} - \max \{1,\dots,1\} = \max \{a_1,\dots, a_m \} - 1.$$ The corresponding monomial is $x^{a_1-1}\dots  x_m^{a_m-1}$.  Moreover, every other monomial that can appear in the sum is of the form 
\[
x^{a_1-1}\dots x_m^{a_m-1} \cdot (x_1 \dots x_m)^{\mu}
\]
 
for some $\mu \in \mathbb{Z}_{>0}$. We can introduce up to $m$ new colors while staying surjective, i.e., that the labeling uses all colors from $1$ to the highest used color. Therefore,
\[
M_{1,i} = \sum_{\mu = 0}^{m - \max \{a_1,\dots, a_m\} +1} x_1^{a_1-1} \cdot \dots \cdot x_m^{a_m-1} \cdot (x_1\cdot \dots \cdot x_m)^{\mu}\text{.}
\]
 
Now we need to adjust this setting to an arbitrary base state $s^i$ corresponding to a vector $(b_1,\dots,b_m)$ where $\min \{ b_i\} = 1$. Target space might be represented by a different vector $\tilde{a} = a + (\lambda, \dots, \lambda)$. We want that 
\[
\min \left( \left\{\tilde{a}_1 -b \colon b\in S_{<_{\tilde{a}_1}} \right\}\cup \left\{\tilde{a}_m -b \colon b \in S_{<_{\tilde{a}_m}} \right\}\right) =0.
\]
However, this is equivalent to 
\[
\lambda = - \min \left(\left\{a_1 -b\colon b \in S_{< a_1} \right\}\cup \dots \cup\left\{ a_m - b\colon  b\in S_{< a_m}\right\}\right).
\]
 
This also gives the (by degree) smallest transition monomial, so we only need to figure out how many transitions from $s^i$ to $s^j$ there are. This corresponds to the question of how many colors we introduce from $(b_1,\dots, b_m)$ to $(\tilde a_1, \dots, \tilde a_m)$. However, the number of colors introduced is 
\[
\max \{ \tilde a_1, \dots, \tilde a_m\} - \max \{b_1,\dots, b_m \},
\]
 
which implies that the highest monomial appearing is
\[
x_1^{\tilde a_1 -b_1}\cdot \dots x_m^{\tilde a_m -b_m} \cdot (x_1 \cdot \dots \cdot x_m)^{\mu_{\max}} = x_1^{a_1 -b_1}\cdot \dots x_m^{a_m -b_m} \cdot (x_1 \cdot \dots \cdot x_m)^{\mu_{\max}+ \lambda},
\] 
where
\[
\mu_{\max} = m - \max \{a_1,\dots, a_m\} -\lambda + \max \{b_1, \dots, b_m\}.
\]
\end{proof}

\bibliographystyle{amsalpha}
\bibliography{diss_refs}

\end{document}